\documentclass[11pt]{amsart}

\usepackage{amsmath,amsthm}
\usepackage{amssymb}
\usepackage{color}
\usepackage{enumerate}

\usepackage{graphicx}
\usepackage{bm}
\usepackage{soul,xcolor}

\newtheorem{thm}{Theorem}[section]
\newtheorem{cor}[thm]{Corollary}
\newtheorem{lem}[thm]{Lemma}
\newtheorem{prop}[thm]{Proposition}

\newtheorem{remark}[thm]{Remark}

\newtheorem{assump}[thm]{Assumption}

\newtheorem{corollary}[thm]{Corollary}
\numberwithin{equation}{section}

\usepackage{hyperref}\usepackage{hyperref}
\def\span{\textrm{ span }}

\def\C{ {\mathbb C} }
\def\R{ {\mathbb R} }
\def\H{ {\mathcal{H}} }
\def\Z{ {\mathbb Z} }

\def\Rd{ {\mathbb R}^d }

\def\T{ {\mathbb T} }

\DeclareMathOperator{\Tr}{Tr}

\begin{document}

\setstcolor{red}

\title %[Phase Retrieval of Complex and Vector-valued Signals]
{Phase retrieval of complex and vector-valued functions}

\author{Yang Chen, Cheng Cheng  and Qiyu Sun}

\address{Chen: Key Laboratory of Computing and Stochastic Mathematics (Ministry of Education), School of Mathematics and Statistics,
 Hunan Normal University, Changsha, Hunan 410081, P. R. China,
 email: ychenmath@hunnu.edu.cn}
\address{Cheng:  Department of Mathematics, Duke University and Statistical and
Applied Mathematical Sciences Institute (SAMSI), Durham, NC 27708,
email:  cheng87@math.duke.edu}

\address{Sun: Department of Mathematics, University of Central Florida, Orlando, FL 32816, email: qiyu.sun@ucf.edu}

\thanks{This project is partially supported by  the National Science Foundation (DMS-1638521 and DMS-1816313), Hunan Province Science Foundation for Youth (2018JJ3329) and  Scientific Research Fund of Hunan Provincial Education Department(18C0059). }
\maketitle
\begin{abstract}
 The phase retrieval problem  in the  classical  setting is to \mbox{reconstruct}  real/complex functions from the magnitudes of their
 Fourier/frame measurements.
In this paper, we consider a new  phase retrieval paradigm in the complex/quaternion/vector-valued  setting,
and we provide  several characterizations to determine complex/quaternion/vector-valued functions  $f$ in a  linear space ${\mathcal S}$
 of (in)finite dimensions,  up to a trivial ambiguity, from the magnitudes  $\|\phi(f)\|$ of their linear measurements $\phi(f), \phi\in \Phi$.
%As an application, we extend the well known  equivalence between the complement property for linear  measurements $\Phi$ and the phase
%retrieval of  linear space ${\mathcal S}$ in the real scalar setting.} {\color{red}
Our characterization in the scalar setting implies
 the well-known equivalence between the complement property for linear  measurements $\Phi$ and the phase
retrieval of  linear space ${\mathcal S}$.
In this paper, we also discuss  the affine phase retrieval of vector-valued functions in a linear space
  and  the reconstruction of vector fields on a graph, up to an orthogonal matrix,
from their absolute magnitudes at vertices and relative magnitudes between neighboring  vertices.
 \end{abstract}

 \section{Introduction}%in the most general formulation of $f$ of some transform naturally
Phase retrieval arises in  various engineering fields, such as X-ray crystallography, coherent diffractive imaging and optics
\cite{F78,    hayes80,  H01,  millane90}. %rabiner93,
  The classical  phase retrieval problem is to recover real/complex functions from the magnitudes of their Fourier measurements.
 Starting from the pioneering work  \cite{BCE06} by  Balan,   Casazza and Edidin,  phase retrieval of
 real vectors  ${\bf x}\in  \R^d$ (or complex vectors ${\bf x}\in \C^d$)
 from the magnitudes  ${\bf y}=|{\bf A} {\bf x}|$ of their frame measurements
 has received considerable attention, where
 ${\bf A}$ is a measurement  matrix, see \cite{ABFM14,BBCE09,  Bandeira14, %candes13,  candes15,
 CSV12, GSXW, jaganathany15, bcjhl19, % netrapalli15,
    schechtman15, W14, WX17}  for historical remarks and  additional references.
    The  phase retrieval paradigm has been recently extended to
    infinite-dimensional setting, where a core problem
   is to recover real/complex functions  $f$ in a linear space of (in)finite dimensions, such as the Paley-Wiener space and shift-invariant spaces,
 from the magnitudes   $|\phi(f)|$ of their  linear measurements  $\phi(f), \phi\in \Phi$, see
      \cite{ADGY16, RGrohs17,  CCD16, YC16,
   cjs17, CS18, grohs17, mallat15,  Mcdonald, pohl2014,  schechtman15,  shenoy16, Thakur11}.
  %  \cite{ADGY16, RGrohs17, CCD16, YC16,  cjs17, CS18, grohs17, mallat15, pohl2014,    shenoy16, Thakur11}.
   The phase retrieval in the infinite-dimensional setting is  fundamentally different from the  finite-dimensional setting
   \cite{RGrohs17, CCD16,YC16, grohs17}, for instance,  phase retrieval
    in an infinite-dimensional Hilbert space is  coherently  unstable \cite{CCD16}, and
    the set of phase retrieval  functions in
    a real shift-invariant space   $S(\psi)$ generated by a compactly supported function $\psi$
    is observed to be neither a convex subset of $S(\psi)$ nor its closed subset  \cite{YC16, cjs17, CS18}.
%not all functions in a shift-invariant space generated by a compactly support function are phase retrievable

% However the phase retrieval problem in . The phaseless reconstruction of real/complex functions  $f$ in a linear space of (in)finite dimensions,
%such as the Paley-Wiener space and shift-invariant spaces, from the magnitudes of their  linear measurements  $\phi(f), \phi\in \Phi$ has been studied,
% see  \cite{ADGY16,
%   cjs17, CS18, mallat15,  Mcdonald, pohl2014,   shenoy16, Thakur11}.
%The above  phase retrieval paradigm has been extended to recover real/complex functions  $f$ in a linear space of (in)finite dimensions,
% such as the Paley-Wiener space and shift-invariant spaces,
% from the magnitudes of their  linear measurements  $\phi(f), \phi\in \Phi$,   see
%
%We remark that
%the phase retrieval paradigm in infinite-dimensional setting  is  fundamentally different from the  finite-dimensional setting.

% In the recovery of a finite-dimensional complex vector  from its absolute values obtained by some real linear measurement process,
%  one can verify that its conjugacy  has the same phaseless linear measurements.
%  The authors in \cite{Lai} consider the conjugate phase retrieval of finite-dimensional complex vectors
%  where  conjugacy and a global phase factor are viewed as trivial ambiguity.
%  In this paper, we treat  orthogonal transformations as trivial ambiguity in determining the $\mathcal H$-valued signal from
%its phaseless measurements. This coincides with the conjugate phase retrieval in the complex setting.
%

Quaternions form a  noncommutative division algebra,  and they are %widely
used in  navigation,  robotics and computer vision
 \cite{altmann86, kuipers98}.
In the first part of this paper, we consider  the phase retrieval problem  whether
a complex/quaternion function $f$ on a domain $D$ in
a linear space of (in)finite dimensions is determined, up to a  unimodular constant and conjugation, from the magnitudes  $|f(x)|, x\in D$.

%  In the real-valued setting, the missing information is $\pm 1$ and it becomes more complicated in the complex-valued setting
% with missing information $e^{i\theta}, \theta\in [0,2\pi)$. Phase retrieval of complex-valued signal is fundamentally different as real
%signals, and it is natural to view a complex valued signal as a vector composed of real and imaginary part of it, and  having the same magnitude.
 %Obviously there exists a bijection between the complex-valued signal $f$ and $\R^2$-valued signal $(\Re f, \Im f)$.

An important problem in the dynamic of a fleet of autonomous mobile robots (AMR) is  to determine  the velocity of each  AMR from  the absolute speed of AMRs and the relative
speed  between neighboring  AMRs.   %Autonomous mobile robots
 As the topology of an AMR fleet can be described by a graph ${\mathcal G}=(V, E)$,  the  velocity  recovery
problem for an AMR fleet becomes whether a vector field  ${\bf f}=({\bf f}_i)_{i\in V}$ on the graph ${\mathcal G}$ can be reconstructed, up to
  an orthogonal matrix,  from its absolute magnitudes
$\|{\bf f}_i\|$ at all vertices $i\in V$ and relative magnitudes $\|{\bf f}_i-{\bf f}_j\|$
of neighboring vertices $(i,j)\in E$.
 In the second part of this paper, we consider  the phase retrieval problem  whether
a vector-valued function $f$ %on a domain $D$
 in
a real  %(in)finite-dimensional
linear space ${\mathcal S}$ of (in)finite dimensions is determined, up to a unitary transformation, from the magnitudes
$\|\phi(f)\|$  of their linear measurements $ \phi(f), \phi\in \Phi$,  and also the reconstruction of a vector field on a  undirected graph,
 up to an orthogonal matrix,
from its  absolute magnitudes at vertices and relative magnitudes between neighboring  vertices.
Our characterization in the real scalar setting  implies
the well-known equivalence between the complement property for linear  measurements $\Phi$ and the phase retrieval of
 linear space ${\mathcal S}$ \cite{RGrohs17, BBCE09, BCE06, CCD16}.

 Affine phase retrieval arises in holography, data separation and phaseless sampling
\cite{CMVG2010,  YC16, DK2013,  DH2001, LBCMDU2003,  LLS2013},
and one of its core problems is whether real/complex  functions are determined uniquely from  their magnitudes of affine linear measurements
 \cite{GSXW, HX18,  KST95, CLS19,
 Mcdonald}.  In the  third part of this paper, we study
affine phase retrieval of vector-valued functions. %and show that for

 This paper is organized as follows. In Sections \ref{complexpr.section} and \ref{pr.vector.sec},
 we consider  the phase retrieval  of complex/quaternion functions and vector-valued functions respectively.
  In Section \ref{affine.pr.sec}, we discuss the affine phase retrieval of vector-valued functions.
  All proofs are collected in Section  \ref{proofs.section}.

\section{Phase retrieval  of complex/quaternion functions}
\label{complexpr.section}

\smallskip

 Let  ${\mathcal C}$  be a complex linear space of  functions  $f$ on a domain $D$ that is  {\em invariant under complex conjugation}, i.e. $
 \bar f\in \mathcal C$ for all $f\in \mathcal C$.
 %And we say ${\mathcal C}$ is {\em  complex conjugate invariant} if it is invariant under complex conjugation.}  \st{We say that ${\mathcal C}$} is  {\em invariant under complex conjugation} if
%$ \bar f \in {\mathcal C}$ for all $f\in {\mathcal C}$.
 Our representative examples of complex conjugate invariant spaces are
the  complex range space
\begin{equation}\label{complexrange.def}
R_{\C}({\bf A})= \big \{{\bf A}{\bf x}, \ {\bf x}\in  \C^n\big\}
\end{equation}
of a real matrix  ${\bf A}$ of size  $m\times n$,  and
the complex shift-invariant space
\begin{equation}\label{complexsis.def}
S_{\C}(\psi)=\Big\{\sum_{k\in \Z} c(k) \psi (\cdot-k), \ c(k)\in \C \ {\rm for \ all} \  k\in \Z\Big\}
\end{equation}
generated by a real-valued function $\psi$ on the real line $\R$ \cite{AG01, AST05, DDR94}.
%and  the range space
%\begin{equation}\label{gaborspace.def}
%V_g=\big\{V_gf(x, \omega),\  f\in L^2(\Rd)\big\}\end{equation}
%of the short-time Fourier transform associated with a real-valued window  function $g\in L^2(\Rd)$,
%where
%$$V_gf(x, \omega)=\int_{\Rd} f(t) g(t-x) \exp(-it\omega) dt, \ (x, \omega)\in \Rd\times \Rd $$
%\cite{grochenigbook, grohs17}.

Given $f\in {\mathcal C}$, let
\begin{equation}\label{complexpr.def}
{\mathcal M}_f:=\{g\in {\mathcal C}, \   |g(x)|=|f(x)| \ {\rm for \ all} \  x\in D\}
\end{equation}
contain all  functions $g\in {\mathcal C}$  that have the same magnitude measurements as the original function $f$ has on the whole domain $D$.
Clearly, %we have
\begin{equation}
{\mathcal M}_f\supset  \{ z f\in {\mathcal C}, \ z\in \T\}\cup\{  z \bar f\in {\mathcal C}, \ z\in \T\},
\end{equation}
where $\T=\{z\in \C, |z|=1\}$.
We say that a function  $f\in {\mathcal C}$ is {\em complex conjugate phase retrieval} in ${\mathcal C}$ if
\begin{equation}\label{complexpr.def}
{\mathcal M}_f= %\{ z f,  z \bar f\in {\mathcal C}, \ z\in \T\}
\{ z f\in {\mathcal C}, \ z\in \T\}\cup\{  z \bar f\in {\mathcal C}, \ z\in \T\},
\end{equation}
and that  the  linear space ${\mathcal C}$ is {\em  complex conjugate  phase retrieval}  if every function $f\in {\mathcal C}$ is
 complex conjugate  phase retrieval in ${\mathcal C}$, see \cite{Lai, LLW19,  Mcdonald}  for complex conjugate phase retrieval  of  vectors in a finite-dimensional
linear  space and of entire functions.
%In this section,
%we consider the following fundamental problem about complex conjugate phase retrieval for signals in a complex conjugate invariant space:

% \begin{problem} Let the complex linear space ${\mathcal C}$ be invariant under complex conjugation.
% Determine whether all complex functions in ${\mathcal C}$
%  can be reconstructed, up to global phase  factor and complex conjugacy, from
% their magnitudes   on the whole domain $D$, i.e.,
% \begin{equation}\label{complexpr.def00}
%{\mathcal M}_f= \{ z f,  z \bar f\in {\mathcal C}, \ z\in \T\} \ {\rm for \ all}\    f\in {\mathcal C}.
%\end{equation}
% \end{problem}

In Section \ref{cjpr.section}, we  characterize complex conjugate phase retrieval of functions
$f$ in a complex conjugate invariant space ${\mathcal C}$, see Theorems \ref{complexpr.thm} and \ref{complexpr.thm2}.
%As an application of  Theorems \ref{complexpr.thm2} and \ref{complexpr.thm},
%we characterize  complex conjugate phase retrievability of
%the complex range space $R_{\C}({\bf A})$  in \eqref{complexrange.def}, see
% Corollary \ref{corollary2.4}
% and  \cite[Theorem 2.3]{Lai} for an equivalent formulation.
 %Let ${\bf A}$ be a real matrix with full rank $n$ and
 Write ${\bf A}^T=({\bf a}_1, \ldots, {\bf a}_m)$.
As an application of Theorem \ref{complexpr.thm2}, we %characterize complex conjugate phase retrievability of
%the complex range space $R_{\C}({\bf A})$  in \eqref{complexrange.def}
show  that the complex range space
${ R}_{\C}({\bf A})$ in \eqref{complexrange.def}
is complex conjugate phase retrieval if and only if %{\color{red} for any $Q\in{\mathcal M}_{n,2}(\R)$, $\Tr({\bf a}_i{\bf a}_i^TQ)=0$ for all $1\le i\le m$
%implies that $Q^T=-Q$, where ${\mathcal M}_{n,2}(\R):=\{Q\in \R^{n\times n }, {\rm rank}(Q)\le 2\}$,}
%{\color{red}(}
there does not exist
a real matrix ${\bf X}$ of rank at most 2 such that
${\bf X}^T\ne  -{\bf  X}$ and $\Tr( {\bf a}_i {\bf a}_i^T  {\bf X} ) =0$ for  all $1\le i\le m$,
  see  Corollary \ref{corollary2.4} and \cite[Theorem 2.3]{Lai} for an equivalent formulation.

% We say that a complex linear space  ${\mathcal C}$ of  functions  $f$ on a domain $D$ is  {\em invariant under complex conjugation} if
%\begin{equation} z f, z \bar f \in {\mathcal C} \ \ {\rm for \ all} \  f\in {\mathcal C}\ {\rm and} \ z\in \C.
%\end{equation}
%Our representative examples  % of complex conjugate invariant spaces
% are
%the  complex range space
%\eqref{complexrange.def}
%of a real matrix
%and
%the complex shift-invariant space
%\eqref{complexsis.def}
%generated by a real-valued function.

Let
$\Re({\mathcal C})$ be the linear space of all real-valued functions in ${\mathcal C}$. By  the complex conjugate
invariance  of the complex linear space ${\mathcal C}$, we have
\begin{equation}\label{realspace.def}
\Re({\mathcal C})=\{(f+\bar f)/2, \ f\in {\mathcal C}\}.
\end{equation}
%is also the linear  space of real parts of functions in ${\mathcal C}$.
 %Let ${\mathcal R}$ be a real linear space of functions on a domain $D$.
 We say that a function  $f \in \Re({\mathcal C})$
 is {\em phase retrieval} in $\Re({\mathcal C})$
  if  % the set
  $${\mathcal M}_f:=\{g\in \Re({\mathcal C}), \   |g(x)|=|f(x)| \ {\rm for \ all} \   x\in D\}$$
  contains only two elements $\pm f$,
and that the  whole real  linear subspace $\Re({\mathcal C})$ is {\em  phase retrieval} if every function $f\in \Re({\mathcal C})$ is  phase
 retrieval in $\Re({\mathcal C})$, i.e., $|f(x)|=|g(x)|, x\in D$ for $f, g\in \Re({\mathcal C})$ only when  $f=\pm g$,
 \cite{ADGY16, YC16, cjs17, CS18, lh18, mallat15, sun17,   WAM15}.
 In Section \ref{ccprrpr.section}, we show that a complex conjugate phase retrieval   function in ${\mathcal C}$ has its
 real part being  phase retrieval in  $\Re({\mathcal C})$, and hence
 $\Re({\mathcal C})$ is  phase retrieval if
  ${\mathcal C}$
is complex conjugate phase retrieval, see Theorem \ref{necess.charac} and Corollary \ref{necess.charac.cor}.
% {\color{red}\Large(}We say that   a function  $f$ in   a real linear space  ${\mathcal R}$ of functions on a domain $D$
% is {\em real phase retrieval} in ${\mathcal R}$
%  if  % the set
%  $${\mathcal M}_f:=\{g\in {\mathcal R}, \   |g(x)|=|f(x)| \ {\rm for \ all} \   x\in D\}$$
%  contain only two elements $\pm f$,
%and that  the  whole linear space ${\mathcal R}$ is {\em real phase retrieval}  if every function $f\in {\mathcal R}$ is
% real phase retrieval in ${\mathcal R}$  \cite{ADGY16, YC16, cjs17, CS18, mallat15, sun17,   WAM15}.
%% The real phase retrieval of functions in a real linear space has been studied in \cite{YC16, cjs17, CS18, sun17}.
% In Section \ref{ccprrpr.section}, we show that a complex conjugate phase {\color{red}retrievable} \st{retrieval}  function in ${\mathcal C}$ has its
% real part being real  phase retrieval in  $\Re({\mathcal C})$, and hence
% $\Re({\mathcal C})$ is real phase retrieval if
%  ${\mathcal C}$
%is complex conjugate phase retrieval, see Theorem \ref{necess.charac} and Corollary \ref{necess.charac.cor}.{\color{red}\Large)}
As an application of Theorem \ref{necess.charac}, we find all complex conjugate phase retrieval functions in
a complex shift-invariant space $S_\C(h)$ %\eqref{complexsis.def}
generated by the hat function
$h(t)=\max(1-|t|, 0)$, see Proposition \ref{complex.sis.pr}.

 The set
$\mathcal Q_8=\{a+b{\bf i}+c{\bf j}+d{\bf k},\ a,\ b,\ c,\ d\in\R\}$ of quaternions
forms a {\em noncommutative} division algebra
 with  the  multiplication rule for the basis $1, {\bf i}, {\bf j}$ and ${\bf k}$ given by
$$ {\bf i}^2={\bf j}^2={\bf k}^2=-1, \ {\bf ij}=-{\bf ji}={\bf k},\ {\bf jk}=-{\bf kj}={\bf i} \ {\rm and} \  {\bf ki}=-{\bf ik}={\bf j}.$$
Quaternions  provide a more compact,  numerically stable, and  efficient representation of orientations and rotations of objects
in  $\R^3$, and they have  many engineering applications such as  navigation,  robotics and computer vision
 \cite{altmann86, kuipers98}.
To our  knowledge,   phase  retrieval  of
quaternion-valued functions on a domain  $D$ has not been discussed in literature.
  In  Section  \ref{quaternionpr.section},
 we   generalize the conclusions in  Theorems \ref{complexpr.thm}, \ref{complexpr.thm2}  and  \ref{necess.charac}
 for complex functions to  quaternion-valued functions.

%the characterization in Theorems \ref{complexpr.thm} and \ref{complexpr.thm2} for complex-valued functions to
%quaternion-valued functions on a domain.

\subsection{Complex conjugate phase retrieval}\label{cjpr.section}  For functions in a complex conjugate invariant linear space, we have
 the following characterization to the complex conjugate phase retrieval,  see Section \ref{complexpr.thm.section} for the proof.

\begin{thm}\label{complexpr.thm}  Let ${\mathcal C}$ be a complex linear space of functions on a domain $D$
that is invariant under complex conjugation, and let $f\in {\mathcal C}$. Then $f$ is complex conjugate phase retrieval
if and only if there do not  exist $u, v\in {\mathcal C}$
satisfying
\begin{equation}\label{complex.decom.1}
f=u+v,
\end{equation}
\begin{equation}\label{complex.decom.2}
 \Re(u(x) \bar v(x))=0\ \
\mbox{\rm for\  all}\ x\in D,
\end{equation}
and
\begin{equation}\label{complex.decom.3}
\Re(u(x) \bar v(y)+u(y) \bar v(x))\ne 0 \ \
\mbox{\rm for\ some}\ x, y\in D.
\end{equation}
\end{thm}

%\begin{proof}
% Write $f=\Re f+ i \Im f:=f_1+if_2$. Take $\tilde{f}=(f_1,f_2)^T$ and $\Phi=\{\delta_x, x\in D\}$. The above corollary follows immediately by  Theorem \ref{hilbertpr.thm}.

%Not all complex linear spaces  are complex conjugate {\color{red}{\em phase retrieval}}. For instance,
% the linear space  $P_n, n\ge 2$, of all complex polynomials of degree at $n$ are  not  complex conjugate {\color{red}{\em phase retrieval}}, since
%  polynomials $p_1(x)=(x-z_1)(x-z_2)$ and $p_2(x)=(x-\bar z_1)(x-z_2)$ with $\Im(z_1)\Im(z_2)\ne 0$
%have same  magnitude on the real line, while
% $p_1(x)\ne zp_2(x)$ or $p_1(x)\ne z\overline p_2(x)$ for all $z\in\T$.
 We remark that
 the set of all complex conjugate phase retrieval  functions in ${\mathcal C}$ is a line cone, but it may be  neither closed  nor convex, cf. \cite{YC16}.
%  complex conjugate phase retrieval in $\mathcal C$ provided that $f$ is complex conjugate phase retrieval in $\mathcal C$. But the set of complex conjugate
%phase retrievable functions in $\mathcal C$ is not a convex subset of $\mathcal C$.}
Applying  Theorem \ref{complexpr.thm}, we
have   the following characterization to
  the complex linear space  $\mathcal C$.

\begin{thm}\label{complexpr.thm2}
    Let ${\mathcal C}$ be a complex linear space of functions on a domain $D$ that is invariant under complex conjugation.
   %Let ${\mathcal C}$ be a complex linear space of functions on a domain $D$ that is invariant under complex conjugation.
Then ${\mathcal C}$ is complex conjugate phase retrieval %i.e., \eqref{complexpr.def00} holds,
if and only if there do not  exist $u, v\in {\mathcal C}$
satisfying \eqref{complex.decom.2} and
\eqref{complex.decom.3}.
\end{thm}

%The conclusions in   follow directly from  Theorem \ref{complexpr.thm}.

%, we have the following result about complex conjugate phase retrievability of a complex linear space. %, see Section \ref{complexpr.thm.section} for the proof.

Now we consider the application of Theorems \ref{complexpr.thm} and  \ref{complexpr.thm2} to
the  complex range space  $R_{\C}({\bf A})$
of  a real matrix  ${\bf A}$ of size $m\times n$.
Without loss of generality, we assume that
\begin{equation}\label{rankassumption}
{\rm rank}({\bf A})=n,\end{equation}
otherwise replacing ${\bf A}$ by its submatrix $\tilde {\bf A}$ of full rank such that $R_{\C}(\tilde {\bf A})=R_{\C}({\bf A})$.
Write ${\bf A}^T=({\bf a}_1, \ldots, {\bf a}_m)$.
As  the vector $f\in R_{\C}({\bf A})\subset \R^m$ can be considered as a function on $D=\{1, \ldots, m\}$,  we obtain %{\bf A}{\bf A} {\bf A}
$${\mathcal M}_f= \big\{ {\bf A} {\bf x}, \  |{\bf a}_i^T {\bf x}|= |{\bf a}_i^T {\bf x}_f|, 1\le i\le m\big\},$$
where  ${\bf x}_f$ is the unique vector in $\C^n$ by \eqref{rankassumption} such that
\begin{equation}\label{xf.def} f={\bf A}{\bf x}_f.\end{equation}
Therefore a function  $f\in R_{\C}( {\bf A})$ is complex conjugate phase   retrieval in $R_{\C}({\bf A})$ if and only if
\begin{equation}
{\mathcal M}_f= \{ z {\bf A} {\bf x}_f,  z\in \T\}\cup \{ z {\bf A}\bar {\bf x}_f,  z\in \T\}.
\end{equation}
Observe that the linear space of all real symmetric matrices is spanned by
${\bf a}_i {\bf a}_j^T+ {\bf a}_j {\bf a}_i^T, 1\le i, j\le m$, by \eqref{rankassumption}.
Hence by  Theorem \ref{complexpr.thm}, we have the following  corollary about complex conjugate phase  retrieval
in $R_{\C}({\bf A})$.

\begin{corollary}\label{corollary2.3new}  Let ${\bf A}$ be  a real matrix of size $m\times n$ satisfying \eqref{rankassumption}. Then $f\in
R_{\C}( {\bf A})$ is  complex conjugate phase retrieval if and only if there do not exist  ${\bf x}_1, {\bf x}_2\in \C^n$ such that
${\bf x}_f={\bf x}_1+{\bf x}_2$, ${\bf X}^T\ne -{\bf X}$, and
$\Tr( {\bf a}_i {\bf a}_i^T {\bf X} )= 0$ for all $1\le i\le m$,
where ${\bf x}_f$ is the unique vector in $\C^n$ satisfying \eqref{xf.def} and
\begin{equation}\label{X.def}
{\bf X}=
\frac{1}{2} ({\bf \bar x}_2 {\bf x}_1^T+ {\bf x}_2 {\bf {\bar x}}_1^T)= \Re({\bf x}_2) \Re ({\bf x}_1^T)+ \Im({\bf x}_2) \Im ({\bf x}_1^T) .\end{equation}
\end{corollary}

Let ${\bf M}_{n,2}(\R)$ be the set of all  real matrices of size $n\times n$ with rank at most 2.
 Observe that  a real matrix ${\bf X}$ is of the form \eqref{X.def} if and only if  ${\bf X}\in {\bf M}_{n,2}(\R)$.
 Therefore
we have the following characterization to complex conjugate phase retrieval of  the complex range space
$R_{\C}({\bf A})$, see \cite[Theorem 2.3]{Lai} for an equivalent formulation  and \cite[Theorem 2.1]{WX17} for phase retrieval of the real linear subspace of $R_{\C}({\bf A})$.

\begin{corollary}\label{corollary2.4}  Let ${\bf A}$ be  a real matrix of size $m\times n$ satisfying \eqref{rankassumption}. Then
the complex range space
$R_{\C}({\bf A})$ is complex conjugate phase retrieval if and only if  there does not exist
${\bf X}\in {\bf M}_{n,2}(\R)$  such that
${\bf X}^T\ne  -{\bf  X}$ and
$\Tr( {\bf a}_i {\bf a}_i^T  {\bf X} ) =0$ for  all $1\le i\le m$.
\end{corollary}

%Next we consider the application of Theorems \ref{complexpr.thm} and  \ref{complexpr.thm2} to
%the  space $V_g$ in \eqref{gaborspace.def}.
%For a real-valued window function $g$.
%Then $V_g$ is a reproducing kernel Hilbert space and it is invariant under complex conjugation. When the window $g$ is compactly supported or Gaussian window, the range space is not complex conjugate phase retrieval. As there exist $f_1, f_2\in L^2(\R)$ such
%that $\supp W_gf_1\cap \supp W_g f_2=\emptyset$ in those cases \cite{ADGY16}, which implies  that \eqref{complex.decom.2} and
%\eqref{complex.decom.3} hold for $V$.

%For a real matrix $X$  of the form \eqref{X.def} that satisfies
%\begin{equation}
%\Tr( ({\bf a}_i {\bf a}_j^T+ {\bf a}_j {\bf a}_i^T)  X ) \ne 0\ \ {\rm for \ all} \ 1\le i, j\le m,
%\end{equation}
%one may verify that
%\begin{equation}
% \langle {\bf x}_1+{\bf x}_2, a_i\rangle \overline{ \langle  {\bf x}_1+{\bf x}_2, a_j\rangle} =
% \langle {\bf x}_1-{\bf x}_2, a_i\rangle \overline{\langle {\bf x}_1-{\bf x}_2, a_j\rangle }, \ 1\le i,j\le m.
%\end{equation}
%This together  with   implies that
%\begin{equation}
%|\langle {\bf x_1}+{\bf x}_2, {\bf a}\rangle|= |\langle {\bf x_1}-{\bf x}_2, {\bf a}\rangle|\ {\rm for \ all} \ {\bf a}\in \R^n.
%\end{equation}

\subsection{Complex conjugate phase retrieval and  phase retrieval}\label{ccprrpr.section}
 Let
$\Re({\mathcal C})$ be the linear space of all real-valued functions in ${\mathcal C}$. By \cite[Theorem 2.1]{YC16},  a real function $f\in \Re({\mathcal C})$ is phase retrieval in $ \Re({\mathcal C})$ if and only if there do not exist
nonzero functions $f_1, f_2\in \Re({\mathcal C})$ such that
\begin{equation}\label{realnonseparable}
f=f_1+f_2 \ \ {\rm and}\ \ f_1 f_2=0.
\end{equation}
The above characterization for phase retrieval in $\Re({\mathcal C})$ is closely related to the characterization
\eqref{complex.decom.1}, \eqref{complex.decom.2} and \eqref{complex.decom.3}  with the requirement \eqref{complex.decom.3} replaced by the nonzero requirement for
the functions  $f_1, f_2\in \Re({\mathcal C})$.
%{\color{red} references and historical remarks about scale case i.e., ${\mathcal H}=\R$.
%In the real-valued setting, $\Phi$ is {\color{red}{\em phase retrieval}} on the real linear space $\mathcal R$
%if and only if $\mathcal M_f=\{\pm f\}$  for all $f\in\mathcal \R$.
 %(c.f. Gabor transform\cite{ADGY16} and wavelet transform\cite{mallat15},\cite{WAM15} on $L^2(\R)$).
 % In the complex setting, $\Phi=\{\delta(x), \delta'(x)\}$ is conjugate phase retrieval on some space of entire signals\cite{Mcdonald}.}
%{\color{red} Observe that the zero function is real phase retrieval in $\Re(\mathcal C)$ and complex conjugate phase retrieval
%phase retrieval in $\mathcal C$. The functions $zf, z\in \C$ are complex conjugate phase retrieval in $\mathcal C$ provided
%that $f$ is complex conjugate phase retrieval.}
In the following
theorem, we introduce  a necessary condition for  a function    to be  complex conjugate phase  retrieval in ${\mathcal C}$,
see Section \ref{necess.charac.section} for the proof.

 \begin{thm}\label{necess.charac}  Let
 ${\mathcal C}$  be  a
 complex linear space of functions on a domain $D$
that is  invariant under complex conjugation, and
$\Re({\mathcal C})$ be its  real linear subspace
in \eqref{realspace.def}. If
a function  $f\in {\mathcal C}$
is   complex conjugate phase retrieval in ${\mathcal C}$, %determined, up to a trivial ambiguity, from their magnitude measurements $|f(t)|,t\in\R$, % if and only if
 then  any function  in the real linear space spanned by the real and imaginary parts of $f$ is  phase retrieval in $\Re({\mathcal C})$.
 \end{thm}

 As an application of the above theorem, we have the following result.

  \begin{cor}\label{necess.charac.cor}  Let
 ${\mathcal C}$   and
$\Re({\mathcal C})$ be  as in Theorem \ref{necess.charac}. If
  ${\mathcal C}$
is complex conjugate phase retrieval,
then $\Re({\mathcal C})$ is  phase retrieval.
 \end{cor}

% Let ${\bf A}$ be a real matrix of size $m\times n$. We say that a linear space $S\subset \R^n$ is {\rm phase retrieval if
% $|{\bf A}{\bf x}|=|{\bf A}{\bf y}|$ holds for ${\bf x}, {\bf y}\in S$ only when  ${\bf y}=\pm {\bf x}$  \cite{lh18}.
% Applying  Theorem \ref{necess.charac} to the complex range space $R_{\C}({\bf A})$  in \eqref{complexrange.def},
% we have the following conclusion.
%
% \begin{cor}\label{matrixprojection.cor} Let ${\bf A}$ be an  $m\times n$ real matrix  satisfying \eqref{rankassumption}. If $f={\bf A} {\bf x}_f\in
% R_{\C}({\bf A})$ is  complex conjugate phase retrieval, then  any  vector  ${\bf x}$ in the space spanned by real and imaginary part of the vector ${\bf x}_f$
% is phase retrievable from its phaseless measurements $|{\bf A}{\bf x}|$.
% %\st{, where $f={\bf A} {\bf x}_f$}.
% \end{cor}

 Let ${\bf A}$ be a real matrix of size $m\times n$ and $\Omega$ be a linear subspace of $\R^n$. We say that the subspace $\Omega$ is {\em phase retrieval} if
 $|{\bf A}{\bf x}|=|{\bf A}{\bf y}|$ holds for ${\bf x}, {\bf y}\in \Omega$ only when  ${\bf y}=\pm {\bf x}$  \cite{lh18}.
 Applying  Theorem \ref{necess.charac} to the complex range space $R_{\C}({\bf A})$  in \eqref{complexrange.def},
 we have the following conclusion about the  phase retrieval of a real linear subspace.

 \begin{cor}\label{matrixprojection.cor} Let ${\bf A}$ be an  $m\times n$ real matrix  satisfying \eqref{rankassumption}.  If $f={\bf A} {\bf x}_f \in
 R_{\C}({\bf A})$ is  complex conjugate phase retrieval, %\st{then  any  vector ${\bf x}$ in the space spanned by real and imaginary part of the vector ${\bf x}_f$
% is phase retrievable from its phaseless measurements $|{\bf A}{\bf x}|$, where $f={\bf A} {\bf x}_f$.}
then the space  spanned by the real and imaginary parts  of the vector ${\bf x}_f$
 is phase retrieval. \end{cor}

 For  a real matrix  ${\bf A}$ of size $m\times 2$, we obtain from Corollary \ref{matrixprojection.cor} that
any vector  ${\bf x}$ in $\R^2$ is phase retrieval  from its phaseless measurement  $|{\bf A}{\bf x}|$ if
$R_{\C}({\bf A})$ is  complex conjugate phase retrieval. In the following proposition, we show that the converse is also true, see Section
\ref{mtwopr.prop.section} for the proof.

\begin{prop}\label{mtwopr.prop}
Let ${\bf A}$ be  a real matrix of size $m\times 2$ satisfying \eqref{rankassumption}.
Then
any vector  ${\bf x}$ in $\R^2$ is phase retrieval from its phaseless measurements  $|{\bf A}{\bf x}|$ if and only if
$R_{\C}({\bf A})$ is conjugate phase retrieval.
\end{prop}

The complex shift-invariant space $S_\C(\psi)$  %in \eqref{complexsis.def}
generated by a compactly supported real function $\psi$ is not complex conjugate phase retrieval,
 as functions $f(t)= z_1 \psi(t)- z_2\psi(t-N)\in S_{\C}(\psi)$ are not  complex conjugate phase retrieval,
 where  $z_1, z_2$ are nonzero complex numbers and $N$ is an integer such that $\psi(t), \psi(t-N)$ have disjoint  supports.
By Theorem \ref{necess.charac}, a necessary condition % {\color{red} for
% a function $f(t)=\sum_{k\in \Z} c(k) \psi(t-k)\in S_{\C}(\psi)$  to be complex conjugate phase retrieval}
 to the complex conjugate phase retrieval of
 a function $f(t)=\sum_{k\in \Z} c(k) \psi(t-k)\in S_{\C}(\psi)$
 is
 that for all $a, b\in \R$,
 $\sum_{k\in \Z} (a \Re c(k) +b \Im c(k)) \psi(t-k)$ are  phase retrieval
 in the real shift-invariant space
 \begin{equation*}
S(\psi)=\Big\{\sum_{k\in \Z} d(k) \psi (\cdot-k), \ d(k)\in \R \ {\rm for \ all} \  k\in \Z\Big\},
\end{equation*}
see \cite{YC16}.
In the following  result, we show that the converse is true when the generator $\psi$ is the hat function
$h(t)=\max(1-|t|, 0)$, see Section \ref{complex.sis.pr.pfsection} for the proof.

%In the following theorem, we apply  Theorem \ref{complexpr.thm} to characterize
% the conjugate phase retrievability for  functions in  the complex shift-invariant space generated by
%  the hat function  $h(t)=\max(1-|t|, 0)$.

%
%\begin{corollary}\label{complex.sis.pr1}
% Let
%\begin{equation*}
%V_{\C}(h)=\Big\{\sum_{k\in \Z} c(k) h (\cdot-k), \ c(k)\in \C \ {\rm for \ all} \  k\in \Z\Big\}
%\end{equation*}
% be the complex shift-invariant space generated by hat function $h(t)=\max(0, 1-|t-1|)$.
%   Then $f\in V_{\C}(h)$%$$f=\sum_{k\in \Z} c(k) \phi (\cdot-k)\in V_{\C}(h)$$
%    is conjugate {\color{red}{\em phase retrieval}} if and only if there does not exist
%     $$ f_1=\sum_{k\in \Z} a(k) h (\cdot-k),\ %{\rm\ and}\
%    f_2=\sum_{k\in \Z} b(k) h (\cdot-k)$$ such that
%\begin{equation}
%f=f_1+f_2,
%\end{equation}
%\begin{equation}
% \Re (a(k) \overline {b(k)})=0 \  {\rm and} \  \Re (a(k) \overline {b(k+1)}+a(k+1) \overline {b(k)})=0,  k\in\Z,
%\end{equation}
%and
%\begin{equation}
%\Re(  a(k) \overline {b(l)}+a(l)\overline{b(k)}) \ne 0\ \ {\rm for \ some} \ k, l\in\Z.
%\end{equation}
%%where ${\bf x}_f$ is the unique vector in $\C^n$ in \eqref{xf.def}.
%\end{corollary}

\begin{prop}\label{complex.sis.pr}
% Let
%\begin{equation*}
%V_{\C}(h)=\Big\{\sum_{k\in \Z} c(k) h (\cdot-k), \ c(k)\in \C \ {\rm for \ all} \  k\in \Z\Big\}
%\end{equation*}
% be the complex shift-invariant space generated by hat function $h(t)=\max(0, 1-|t|)$.
 A nonzero function
   $f=\sum_{k\in \Z} c(k) h (\cdot-k)\in S_{\C}(h)$
 is complex conjugate phase retrieval if and only if
\begin{equation}  \label{complex.sis.pr.eq2}
c(k)\ne 0  \ {\rm for \ all} \ K_--1< k<K_++1\end{equation}
and  there exists at most one $k\in (K_--1, K_+)$ such that
\begin{equation}\label{complex.sis.pr.eq3}
 \Im \big(c(k)\overline{c(k+1)}\big)\ne 0,
 \end{equation}
where $K_-=\inf\{k, c(k)\neq 0\}$ and $K_+=\sup\{k,  c(k)\neq 0\}$.
\end{prop}

\subsection{Quaternion conjugate phase retrieval}\label{quaternionpr.section}
 For a quaternion $q= a+b{\bf i}+c{\bf j}+d{\bf k}\in\mathcal Q_8$, denote its conjugate, real part, and norm by
%The conjugate of $q\in\mathcal Q$ is the quaternion
 $q^{*}=a-b{\bf i}-c{\bf j}-d{\bf k}$, $\Re(q)=a$, and
$ \|q\|={\sqrt {qq^{*}}}={\sqrt {q^{*}q}}={\sqrt {a^{2}+b^{2}+c^{2}+d^{2}}}$
respectively.
A linear space  ${\mathcal W}$ of quaternion-valued functions on a domain $D$
is  {\em quaternion conjugate invariant}
if
\begin{equation}
\label{quaternioninvariant.def} qf,  (qf)^*\in {\mathcal W} \ {\rm \ for \ all} \ q\in {\mathcal Q}_8 \ {\rm and} \ f\in {\mathcal W}.
\end{equation}
For  $f %=f_1+f_2 {\bf i}+f_3 {\bf j}+f_4{\bf k}
 \in \mathcal W$, define
\begin{equation}\label{qunpr.def}
{\mathcal M}_f:=\{g\in {\mathcal W}, \   \|g(x)\|=\|f(x)\| \ \ {\rm for\ all} \ \ x\in D\},
\end{equation}
and  write
\begin{eqnarray}\label{f1234.def}
f & \hskip-0.08in = & \hskip-0.08in  \frac{f+f^*}{2}+\frac{-{\bf i} f-({\bf i} f)^*}{2} {\bf i} +
\frac{-{\bf j} f-({\bf j} f)^*}{2} {\bf j}+ \frac{-{\bf k} f-({\bf k} f)^*}{2} {\bf k}
\nonumber \\
& \hskip-0.09in =: & \hskip-0.08in f_1+f_2 {\bf i}+f_3 {\bf j}+f_4{\bf k}.
\end{eqnarray}
Since the linear space ${\mathcal W}$ is quaternion conjugate invariant, we obtain
$f_i, 1\le i\le 4$, and their quaternion linear combinations
 belong to ${\mathcal W}$,
\begin{equation} q_1 f_1+q_2 f_2+q_3 f_3 + q_4 f_4
\in {\mathcal W}\ {\rm for \ all} \  q_i\in {\mathcal Q}_8, 1\le i\le 4.
\end{equation}
%for all $q_1,  q_2, q_3, q_4\in {\mathcal Q}_8$.
%Clearly,
%$${\mathcal M}_f\supset \{ q(\pm f_1 \pm f_2 {\bf i}\pm f_3 {\bf j}\pm f_4 {\bf k}), q\in {\T}_8\}
%$$
Observe that
\begin{equation*}
\Big\|\sum_{i=1}^4 q_i f_i\Big\|^2= \sum_{i=1}^4 \|q_i\|^2 |f_i|^2+\sum_{1\le i<j\le 4} f_i f_j (q_i q_j^*+q_j q_i^*)
\end{equation*}
for all $q_i\in {\mathcal Q}_8, 1\le i\le 4$. Then
$$
\Big\|\sum_{i=1}^4 q_i f_i\Big\|^2= \sum_{i=1}^4 |f_i|^2= \|f\|^2
$$
if  $q_i\in {\T}_8=\{q\in {\mathcal Q}_8, \|q\|=1\}, 1\le i\le 4$,
 are unit quaternions  satisfying
\begin{equation}\label{qiqj.def}
q_i q_j^*+q_j q_i^*=0\ {\rm for \ all} \ 1\le i<j\le 4.
\end{equation}
 Therefore
\begin{equation*}
{\mathcal M}_f\supset\Big\{ \sum_{i=1}^4 q_i f_i\in  {\mathcal W}, \
q_i\in {\T}_8, 1\le i\le 4, \ {\rm satisfies\ \eqref{qiqj.def}}\Big\}.
\end{equation*}
%
%
%and  unit quaternions $q_i\in {\T}_8= \{q\in {\mathcal Q}_8, \|q\|=1\}$, we have
%\begin{equation}
%\|q f(x)\|= \|q f^*(x) \|=|f(x)|, x\in D,
%\end{equation}
%and
%First we prove that $f_i\in {\mathcal W}$ and then find the set ${\mathcal M}_f$.
%Therefore the set
%\begin{equation}\label{qunpr.def}
%{\mathcal M}_f:=\{g\in {\mathcal W}, \   \|g(x)\|=\|f(x)\| \ {\rm for\ all} \ \ x\in D\}.
%\end{equation}
%of  functions $g\in {\mathcal W}$ having  same magnitude measurements as the original function $f$ has on the whole domain $D$ contains
%\begin{equation}
%{\mathcal M}_f\supset \{  q_1  f_1+ q_2 f_2+ q_3 f_3+ q_4 f_4
%(qq_2, \  q_1, q_2\in {\T}_8\}\cup \{ q_1 f^* q_2,\  q_1, q_2\in {\T}_8\}.
%\end{equation}
%where {\bf not large enough?}
 In this paper, we say that $f\in {\mathcal W}$ is {\em quaternion conjugate phase retrieval} in  ${\mathcal W}$ if
 %the above inclusion becomes an equality, i.e.,
 \begin{equation} \label{qunpr.def}
{\mathcal M}_f=\Big\{ \sum_{i=1}^4 q_i f_i\in  {\mathcal W}, \
q_i\in {\T}_8, 1\le i\le 4, \ {\rm satisfies\ \eqref{qiqj.def}}\Big\},
\end{equation}
and the linear space ${\mathcal W}$ is  {\em quaternion conjugate phase retrieval} if every function in ${\mathcal W}$ is  quaternion conjugate phase retrieval.
In this section, we extend the characterizations in Theorems \ref{complexpr.thm} and \ref{complexpr.thm2} for complex-valued functions to
quaternion-valued functions, see Section \ref{quaternionpf.thm.pfsection} for the proof. % $D$.
%, we have the following characterization to conjugate phase retrievability of a complex linear space.

\begin{thm}\label{quaternionpr.thm}
Let ${\mathcal W}$ be a quaternion linear space of functions  on a domain $D$ which is invariant under quaternion conjugation.
Then  the following statements hold.

\begin{itemize}
\item [{(i)}] $f\in {\mathcal W}$  is quaternion conjugate phase retrieval if and only if there do not exist  $u, v\in {\mathcal W}$ such that
 \begin{equation}  \label{quaternionpr.thm.eq3}
f=u+v,
\end{equation}
 \begin{equation}  \label{quaternionpr.thm.eq4}
\Re( u(x) v^*(x))=0 \ \ {\rm for \ all} \ x\in D,
\end{equation}
and
\begin{equation}  \label{quaternionpr.thm.eq5}
 \Re (u(x) v^*(y) + u(y) v^*(x))\ne 0
 \ \ {\rm for \ some} \   x, y\in D.
\end{equation}

\item[{(ii)}]  The quaternion linear space  ${\mathcal W}$  is quaternion conjugate
phase retrieval if and only if there do not exist  $u, v\in {\mathcal W}$ such that
\eqref{quaternionpr.thm.eq4}
 and \eqref{quaternionpr.thm.eq5}
 hold.
 \end{itemize}

\end{thm}

%\begin{thm}\label{quaternionpr.thm}
%A quaternion linear space  ${\mathcal W}$ which is invariant under quaternion conjugation
%is phase retrieval if and only if there do not exist  $u, v\in {\mathcal W}$ such that
%\eqref{quaternionpr.thm.eq4}
% and \eqref{quaternionpr.thm.eq5}
% hold.
%\end{thm}

For a quaternion conjugate invariant space ${\mathcal W}$ of functions on a domain $D$, let
 \begin{equation}\label{Wr.def00}
 \mathcal W_\R=\{ \Re(f), \ f\in {\mathcal W}\}
 \end{equation}
 be the linear subspace of ${\mathcal W}$ containing all real functions in $\mathcal W$.
For $f=f+0{\bf i}+0{\bf j}+0{\bf k}\in {\mathcal W}_{\R}$,  one may verify that
 $$\Big\{ \sum_{i=1}^4 q_i f_i\in  {\mathcal W}_{\R}, \
q_i\in {\T}_8, 1\le i\le 4, \ {\rm satisfy\ \eqref{qiqj.def}}\Big\}=\{\pm f\}.$$
%\cap {\mathcal W}_{\R}.$$
We say  that a  function $f\in{\mathcal W}_\R$ is {\em phase retrieval} in ${\mathcal W}_\R$ if $\mathcal M_f\cap {\mathcal W}_\R=\{\pm f\}$.
Therefore we have the following result, cf. Corollary \ref{necess.charac.cor}.

\begin{cor}
Let ${\mathcal W}$ be a quaternion linear space of functions on a domain $D$ which is invariant under quaternion conjugation.
If $f\in {\mathcal W}$ is quaternion conjugate phase retrieval, then $\Re(f)$ is  phase retrieval  in ${\mathcal W}_\R$.
\end{cor}

For a quaternion conjugate invariant space ${\mathcal W}$, let % of functions on a domain $D$, let
 \begin{equation}
 \mathcal W_\C=\{ \Re(f) +\Re(-{\bf i}f) {\bf i}, \ f\in {\mathcal W}\}
 \end{equation}
 be the linear subspace of ${\mathcal W}$ that  contains all complex functions in $\mathcal W$.
 For  any complex function $f=f_1+f_2{\bf i}+0{\bf j}+0{\bf k}\in {\mathcal W}_{\C}$ and $z\in \T$,
 one may verify that
%
% $$z (f_1\pm f_2 {\bf i})= z f_1\pm z {\bf i} f_2+ z {\bf j} 0+ z{\bf k} 0$$
% which implies that
   $$ \Big\{ \sum_{i=1}^4 q_i f_i\in  {\mathcal W}_{\C},
q_i\in {\T}_8, 1\le i\le 4, \ {\rm satisfy\ \eqref{qiqj.def}}\Big\}= \{z (f_1\pm f_2 {\bf i}), \ z\in \T\}.$$
We say that a complex function $f\in{\mathcal W}_\C$ is {\em complex conjugate phase retrieval} in ${\mathcal W}_\C$ if
$$\mathcal M_f\cap {\mathcal W}_\C=\{zf, z\in\T\}\cup \{z \bar f, z\in \T\}.$$
 Therefore we have the following result about  complex conjugate phase
retrieval in ${\mathcal W}_\C$.

\begin{cor}
Let ${\mathcal W}$ be a quaternion linear space of functions on a domain $D$ which is invariant under quaternion conjugation.
If $f=f_1+f_2{\bf i}+ f_3{\bf j}+ f_4{\bf k}\in {\mathcal W}$ is quaternion conjugate phase retrieval in ${\mathcal W}$, then $f_1+f_2{\bf i}$ is
complex conjugate phase
retrieval in ${\mathcal W}_\C$.
\end{cor}

\section{Phase retrieval of vector-valued functions}\label{pr.vector.sec}

%{\color{red} General Comments: Throughout this section, we never introduce the relation between $\mathcal H$ and the linear space $S$ explicitly, but in the theorems, we all have the assumption on $\mathcal H$.  }
%

Let ${\mathcal H}$ be  a  real separable Hilbert space with inner product and norm denoted by $\langle \cdot, \cdot\rangle$ and $\|\cdot\|$,
and  ${\mathcal U}(\mathcal{H})$ be the group of unitary operators  $U$ on ${\mathcal H}$
that satisfy
  $$U^*U = UU^* = I,$$
   where $U^*$ is the adjoint of $U$. We say that a function $f$  on a domain $D$ is  {\em $\mathcal H$-valued},   or {\em vector-valued} when unambiguous, if
$$f(x)\in \mathcal{H}  \ \mbox{ for\ all}\ x\in D,$$
and  a real linear space ${\mathcal S}$  of  vector-valued functions if  on a domain $D$ is
 {\em unitary invariant}  if
  \begin{equation}\label{invariant.assumption}
Uf\in {\mathcal S} \  \ {\rm  for \ all }  \ f\in {\mathcal S}\ {\rm and}  \ U\in {\mathcal U}(\mathcal{H}).
\end{equation}
Our representative  unitary invariant linear spaces are
\begin{equation}\label{cr2.def}
{\mathcal C}_{\R^2}=\left\{ \left (\begin{array}{c} \Re(f) \\ \Im(f)\end{array}\right)\!, \   f\in {\mathcal C}\right\}
\end{equation}
associated with a complex  conjugate invariant linear space ${\mathcal C}$,
 the linear space of  vector fields on a spatially distributed network  \cite{bamieh02, cjs17acha,jcs17, motee08, ms17},
 and
vector-valued reproducing kernel spaces in multi-task learning  \cite{arl12, cvtu10, lsz19, mp05}.

We say that
 a family $\Phi $  of linear measurements on ${\mathcal S}$ is  {\em unitary invariant} if
\begin{equation}\label{PhiInvariance.def}
\phi(Uf)=U\phi(f)\in {\mathcal H}\  \ {\rm for \ all} \   f\in {\mathcal S}\ {\rm  and}  \ U\in {\mathcal U}(\mathcal{H}).
\end{equation}
 For a  vector-valued function $f\in {\mathcal S}$ and a unitary invariant family $\Phi$ of  linear measurements,
let
$$\mathcal{M}_{f, \Phi}=\big\{g\in {\mathcal S},  \  \|\phi(g)\|=\|\phi(f)\| \ \ {\rm for \ all} \ \ \phi\in \Phi\big\} $$
contain all vector-valued functions
$g\in {\mathcal S}$  such that $g$ and $f$ have the same magnitude observations.
By the unitary invariance of the  linear space ${\mathcal S}$ and the family $\Phi$ of linear measurements,
 we have
  \begin{equation}\label{vector.mf.eq}
  {\mathcal M}_{f, \Phi}\supset %\{Tf: \ T\in U(\mathcal{H})\}+N_\Phi
  \{Uf, \ U\in {\mathcal U}(\mathcal{H})\}, \ f\in {\mathcal S}.
  \end{equation}
In this paper, we  say that a vector-valued function $f\in  {\mathcal S}$ is  {\em phase retrieval} in  ${\mathcal S}$ if
the inclusion in \eqref{vector.mf.eq} becomes an equality, i.e.,
$${\mathcal M}_{f, \Phi}=\{Uf, \  U\in {\mathcal U}(\mathcal{H})\},$$
 and the unitary invariant space ${\mathcal S}$ is  phase retrieval
 if  every vector-valued function in ${\mathcal S}$ is  phase retrieval, see
 \cite{ADGY16, BCE06, CCD16, YC16, cjs17, CS18, mallat15, sun17, Thakur11,  WAM15} for phase retrieval of   scalar-valued functions in various
 function spaces.  In Section \ref{vectorpr.subsection}, we  characterize  the phase  retrieval  of vector-valued functions in ${\mathcal S}$, see  Theorems \ref{hilbertpr.thm}  and \ref{hilbertpr.thm2}.
 Applying our characterization in the real scalar setting, we have  the well known  equivalence between the complement property
 for linear  measurements $\Phi$ and the phase  retrieval of linear space ${\mathcal S}$ \cite{RGrohs17, BBCE09, BCE06, CCD16},
 see  Corollary \ref{complement.cor2}.

%we consider the phase retrievability of a vector-valued function $f\in {\mathcal S}$ from its magnitude measurements
%$\|\phi(f)\|, \phi\in \Phi$, and in the {\color{red}scalar} setting we {\color{red} revisit}
%the well known  equivalence between the complement property
%for linear  measurements $\Phi$ and the phase  retrievability of the linear space ${\mathcal S}$ (\cite{RGrohs17, BBCE09, BCE06, CCD16}),
%see Theorems  \ref{hilbertpr.thm} and Corollary \ref{complement.cor2}.

Let   $\tilde {\mathcal H}$ be a linear subspace of the  Hilbert space ${\mathcal{H}}$, denote
 the projection from ${\mathcal{H}}$ onto ${\tilde {\mathcal H}}$ by $P_{\tilde {\mathcal H}}$, and set
\begin{equation}\label{PHtildeH.def}
 P_{\tilde {\mathcal H}} {\mathcal S}= \{ P_{\tilde {\mathcal H}}f, \  f\in {\mathcal S}\}.\end{equation}
As an application of Theorems \ref{hilbertpr.thm} and \ref{hilbertpr.thm2}, we show  that the  projection  $P_{\tilde {\mathcal H}} f$
of a phase retrieval  function  $f$ in ${\mathcal S}$  is  phase  retrieval in  the projection  space
$P_{\tilde {\mathcal H}} {\mathcal S}$, see Theorem \ref{subspacePR.thm} in Section  \ref{projective.subsection} and
cf. \cite{CCPW16, edidin17} for phase retrieval by projections.

Let ${\mathcal G}=(V, E)$ be a simple graph and ${\bf f}=({\bf f}_i)_{i\in V}$ be a $d$-dimensional vector field on ${\mathcal G}$, where
${\bf f}_i\in \Rd, i\in V$. In Section
\ref{vector.subsection}, we consider the problem whether  the vector field  ${\bf f}$ can be reconstructed, up to  an orthogonal matrix,
 from its  absolute magnitudes
$\|{\bf f}_i\|$ at all vertices $i\in V$ and relative magnitudes $\|{\bf f}_i-{\bf f}_j\|$
of neighboring vertices $(i,j)\in E$.  In other words,  given any vector field  ${\bf g}=({\bf g}_i)_{i\in V}$ satisfying
\begin{equation}  \label{vectorfield.thm.pfeq1}
\|{\bf g}_i\|=\|{\bf f}_i\| \ \ {\rm for \ all}\ i \in V
\end{equation}
and
\begin{equation}\label{vectorfield.thm.pfeq2}
\|{\bf g}_i-{\bf g}_j\|=\|{\bf f}_i-{\bf f}_j\| \ \ {\rm for \ all}   \ (i,j)\in E,
\end{equation}
 we can find  an orthogonal matrix $U$ of size $d\times d$ such that  ${\bf g}=U{\bf f}$.
%{\color{blue} \bf replaced by in red} In the complete graph setting, we show in Theorem \ref{vectorfieldcompletegraph.thm}  that
%a velocity field ${\bf f}=({\bf f}_i)_{i\in V}$  is determined, up to an orthogonal matrix,
% from its absolute speeds  $\|{\bf f}_i\|$ at all vertices $i\in V$ and relative speeds $\|{\bf f}_i-{\bf f}_j\|$
% between vertices $i,j\in V$.  In the noncomplete graph setting, we introduce a $d$-simplex graph
% ${\mathcal G}_{\bf f}=(V_{\bf f}, E_{\bf f})$ for a vector field
% ${\bf f}$ on ${\mathcal G}$ and show in  Theorem \ref{vectorfield.thm} that
%a $d$-dimensional {\color{red}vector}  \st{velocity} field  ${\bf f}$ on a simple graph  ${\mathcal G}$ is determined, up to an orthogonal matrix,
% from its absolute speeds at vertices and relative speeds between vertices when
% the $d$-simplex graph  ${\mathcal G}_{\bf f}$ is connected.
  For the case that $\mathcal G$ is a complete graph, we show in  Theorem \ref{vectorfieldcompletegraph.thm}
   that a  vector field ${\bf f}=({\bf f}_i)_{i\in V}$ on $\mathcal G$  is determined, up to an orthogonal matrix,
 from its absolute  magnitudes  $\|{\bf f}_i\|$ at all vertices $i\in V$ and relative  magnitudes $\|{\bf f}_i-{\bf f}_j\|$
 between vertices $i,j\in V$. For an arbitrary simple graph $\mathcal G$, given a    vector field
 ${\bf f}$, we introduce a $d$-simplex graph
 ${\mathcal G}_{\bf f}=(V_{\bf f}, E_{\bf f})$ associated with $\bf f$, and  show in Theorem  \ref{vectorfield.thm}
 that $\bf f$  is determined, up to an orthogonal matrix,
 from its absolute  magnitudes at vertices and relative  magnitudes between vertices when
 the $d$-simplex graph  ${\mathcal G}_{\bf f}$ is connected.

% the  space
%\begin{equation}
%V_{\mathcal G}=\{ {\bf v}=({\bf v}_k)_{k\in V},   {\bf v}_k\in {\mathcal H}\}.
%\end{equation}
% of   vector fields on a   graph ${\mathcal G}=(V, E)$.

%For any $T\in U(\mathcal{H})$ and vector-valued function $f\in \mathcal S$,
%it follows from \eqref {invariant.assumption} and \eqref{PhiInvariance.def}
%that $Tf\in {\mathcal S}$ and
%$$\|\phi(Tf+h)\|=\|T\phi(f)\|= \|\phi(f)\|, \phi\in \Phi, $$
%for all $h\in N_\Phi$.
%This implies that % or  equivalently
%  \begin{equation}\label{vector.mf.eq}
%  {\mathcal M}_f\supset \{Tf: \ T\in U(\mathcal{H})\}+N_\Phi=\{Tf: \ T\in U(\mathcal{H})\}, \ f\in {\mathcal S},
%  \end{equation}
%  where the last  equation follows from \eqref{NPhi.zero}.

\subsection{Phase retrieval for vector-valued setting}\label{vectorpr.subsection}
 For a vector-valued function $f\in \mathcal S$,  by   the unitary invariance of  linear space ${\mathcal S}$
 and the family $\Phi$ of linear measurements,   we have
 \begin{equation} \label{unitary.pro}
 \langle \phi(g), \tilde \phi(g)\rangle %=  \langle T\phi(f),  T\tilde \phi(f)\rangle
 =  \langle \phi(f), \tilde \phi(f)\rangle \ \ {\rm for \ all} \ \phi, \tilde \phi\in \Phi,
 \end{equation}
 when $g=Uf$ for some $U\in {\mathcal U}({\mathcal H})$.
 To consider the phase retrieval of vector-valued functions $f$ from their magnitude measurements  $\|\phi(f)\|, \phi\in \Phi$,
  we assume that the converse in \eqref{unitary.pro} holds. % for  the family $\Phi$ of linear  measurements.

 \begin{assump}\label{unitary.assumption}
Let  $f$ be a $\mathcal H$-valued function in the linear space ${\mathcal S}$.  %  $f\in {\mathcal S}$.
Then for any  $\mathcal H$-valued function $g\in {\mathcal S}$ satisfying \eqref{unitary.pro},
there exists $U\in {\mathcal U}({\mathcal H})$ such that $g=Uf$.
 \end{assump}

 A necessary condition for  Assumption \ref{unitary.assumption} to hold is that
   \begin{equation}\label{unitary.injective}
  %\overline{\span\{\phi(f),\phi\in \Phi_0\}}=\mathcal H,\ \forall f\in V.
T_{\Phi}: \ {\mathcal S}\ni  f\mapsto \{\phi(f)\}_{\phi\in\Phi}\rm{\ is\ injective},
  \end{equation}
  or equivalently    the null space of the above map is trivial,
   %complement of the spanning set $\Phi$ is zero, } \st{above map only contains the zero function},
%  $$   \color{red} (\span \Phi )^{\perp}=\{0\}.$$
   \begin{equation}\label{NPhi.zero}
N_{ T_{\Phi}}=\{0\}.\end{equation}
   In the case that ${\mathcal H}$ is finite dimensional,  the injectivity in \eqref{unitary.injective} is also
   sufficient for the Assumption  \ref{unitary.assumption}, see Remark \ref{assump.remark}.

In the next theorem, we  characterize   the phase retrieval of   vector-valued functions in a unitary invariant space, see
  Section  \ref{hilbertpr.thm.section} for the proof.

\begin{thm}\label{hilbertpr.thm}
Let ${\mathcal{H}}$ be a real separable Hilbert space,  ${\mathcal S}$ be a unitary invariant space of  ${\mathcal H}$-valued functions
 on the domain $D$, and $\Phi$ be a unitary invariant set of linear measurements satisfying Assumption  \ref{unitary.assumption}.
Then  $f\in {\mathcal S}$ is  phase retrieval if and only if  there do not exist  $u, v\in {\mathcal S}$ such that
 \begin{equation}  \label{hilbertpr.thm.eq3}
f=u+v,
\end{equation}
  \begin{equation}  \label{hilbertpr.thm.eq4}
\langle \phi(u), \phi(v)\rangle=0 \ \ {\rm for \ all} \ \phi\in \Phi,
\end{equation}
and
\begin{equation}  \label{hilbertpr.thm.eq5}
 \langle \phi_0(u), \phi_1(v)\rangle + \langle \phi_1(u), \phi_0(v)\rangle\ne 0
 \ \ {\rm for \ some} \   \phi_0, \phi_1\in \Phi.
\end{equation}
\end{thm}

Applying Theorem  \ref{hilbertpr.thm}, we have  a characterization to phase retrieval of   a unitary invariant space of vector-valued functions.

\begin{thm}\label{hilbertpr.thm2}
Let ${\mathcal{H}}$ be a real separable Hilbert space,  ${\mathcal S}$ be a unitary invariant space of ${\mathcal H}$-valued functions
 on the domain $D$, and $\Phi$ be a unitary invariant set of linear measurements satisfying Assumption  \ref{unitary.assumption}.
Then the unitary invariant space ${\mathcal S}$ is
 phase retrieval
 if and only if  there do not exist  $u, v\in {\mathcal S}$ such that
  \eqref{hilbertpr.thm.eq4}
and   \eqref{hilbertpr.thm.eq5} hold.
\end{thm}

%  as
%    \begin{equation}\label{vector.mf.eq}
%  {\mathcal M}_f\supset \{Tf: \ T\in U(\mathcal{H})\}+N_\Phi=\{Tf: \ T\in U(\mathcal{H})\}, \ f\in {\mathcal S},
%  \end{equation}
%For any $T\in U(\mathcal{H})$ and vector-valued function $f\in \mathcal S$,
%it follows from \eqref {invariant.assumption} and \eqref{PhiInvariance.def}
%that $Tf\in {\mathcal S}$ and
%$$\|\phi(Tf+h)\|=\|T\phi(f)\|= \|\phi(f)\|, \phi\in \Phi, $$
%for all $h\in N_\Phi$.

\begin{remark}\label{assump.remark} {\rm
 In this remark, we show that in the finite-dimensional setting, i.e., $\dim {\mathcal H}<\infty$,
a unitary invariant set $\Phi$ of linear measurements satisfies Assumption  \ref{unitary.assumption} if and only if the
map  $T_\Phi$ in  \eqref{unitary.injective} is injective. The necessity is obvious since given any  function $g$ in the null
space $N_{T_\Phi}$, we have
$$\langle \phi(g), \tilde \phi(g)\rangle=0 \ {\rm for \ all} \ \phi, \tilde \phi\in \Phi.$$
% since given any $g\in N_\Phi$, we have
%$$\langle \phi(g), \tilde \phi(g)\rangle=0 \ {\rm for \ all} \ \phi, \tilde \phi\in \Phi.$$
Now we prove the sufficiency.  % Let functions $f, g\in {\mathcal S}$ satisfy \eqref{unitary.pro}, and
Let
$W_f$ and $W_g$ be the  linear subspaces of ${\mathcal H}$ spanned by $\phi(f)$ and by $\phi(g), \phi\in \Phi$ respectively.
Then
the space  $W_f$ has  a  basis  $\phi_{n}(f), 1\le n\le N$. Applying
the  Gram-Schmidt procedure to the above basis, % $\phi_{n}(f), n\ge 1$,
 we can construct
   an orthonormal basis
   \begin{equation}  \label {orthogonal.lem.pf.eq3++}
e_m=\sum_{n=1}^m a(m,n) \phi_{{n}}(f),  1\le m\le N
\end{equation}
of the linear space $W_f$, where  $ a(m,n), 1\le n\le m\le N$  are functions  of $\langle \phi_{i}(f), \phi_j(f)\rangle, 1\le i, j\le N$.
 Define
    \begin{equation*}  \label {orthogonal.lem.pf.eq4}
\tilde e_m=\sum_{n=1}^m a(m,n) \phi_{{n}}(g), 1\le m\le N.
\end{equation*}
Then $\tilde e_m, 1\le m\le N$, is an orthonormal basis of the space $W_g$ by \eqref{unitary.pro}
and \eqref{orthogonal.lem.pf.eq3++}.
Define a  unitary operator $U$  on ${\mathcal{H}}$ such that
\begin{equation}  \label {orthogonal.lem.pf.eq5}
Ue_n= \tilde e_n, \ 1\le n\le \dim {\mathcal H},
\end{equation}
where $\{e_n, N+1\le n\le \dim {\mathcal H}\}$ and
 $\{\tilde e_n, N+1\le n\le \dim {\mathcal H}\}$ are orthonormal bases of orthogonal complements of $W_f$ and $W_g$ in ${\mathcal H}$ respectively.
For the above  unitary operator $U$, we have
\begin{equation}
\phi(g)=U\phi(f)= \phi(Uf) \ {\rm for \ all}\  \phi\in \Phi,
\end{equation}
where the last equality follows from  unitary invariance  of the linear measurements $\Phi$.
This together with the injectivity hypothesis proves $g=Uf$ and hence Assumption  \ref{unitary.assumption} holds.
}
\end{remark}

The functions $u$ and $v$ in \eqref{hilbertpr.thm.eq5} must be nonzero functions.
The converse is true for the scalar   setting, i.e., ${\mathcal H}=\R$, since
\begin{eqnarray*} \langle \phi_0(u), \phi_1(v)\rangle + \langle \phi_1(u), \phi_0(v)\rangle
&\hskip-0.08in = & \hskip-0.08in
\phi_0(u) \phi_1(v)+ \phi_1(u)\phi_0(v)\\
&\hskip-0.08in = & \hskip-0.08in  \phi_0(u) \phi_1(v)\ne 0
\end{eqnarray*}
by  \eqref{unitary.injective} and \eqref{hilbertpr.thm.eq4}, where $\phi_0, \phi_1\in \Phi$ are so chosen that
$\phi_0(u) \ne 0$ and $\phi_1(v)\ne 0$.
Therefore by Theorem \ref{hilbertpr.thm},  we have the following result,  which is established in  \cite{YC16}
when  $\Phi$ is the set of  point-evaluation functionals. %on the domain $D$.

\begin{cor}\label{scalepr.cor}
Let  ${\mathcal S}$ be a linear space of real functions
 on the domain $D$, and $\Phi$ be a  set of linear measurements satisfying \eqref{unitary.injective}.
Then
$f\in {\mathcal S}$ is  phase retrieval  if and only if  there do not exist nonzero functions  $u, v\in {\mathcal S}$ such that
\begin{equation}\label{separable.def00}
f=u+v \ \ {\rm and}\  \
 \phi(u) \phi(v)=0 \ \ {\rm for \ all} \ \phi\in \Phi.
\end{equation}
\end{cor}

For a unitary invariant set $\Phi$ of linear measurements satisfying Assumption  \ref{unitary.assumption}, the null space of the map
 \eqref{unitary.injective}    contains the zero element only.
 %\st{trivial, i.e., $N_{\color{red} T_{\Phi}}=\{0\}$}.
  A strong version about the set $\Phi$  is its {\em complement property}
that given any $\tilde \Phi\subset \Phi$,
\begin{equation}
\label{complementary.def}
{\rm either  }\  N_{T_{\tilde\Phi}}=\{0\} \ {\rm  or} \
N_{T_{\Phi\backslash\tilde{\Phi}}}=\{0\}.
\end{equation}
In  the  scalar  setting,  i.e., ${\mathcal H}=\R$, an equivalent formulation of the complement property
\eqref{complementary.def} is that
there do not exist  nonzero functions $u, v\in {\mathcal S}$ such that
$ \phi(u) \phi(v)=0$ for all $\phi\in \Phi$. Therefore by Corollary \ref{scalepr.cor},
we have the following result, which is established in  \cite{RGrohs17, BBCE09, BCE06, CCD16} for frames in Hilbert/Banach space setting.

\begin{cor}\label{complement.cor2}
Let  ${\mathcal S}$ be a linear space of real functions
 on the domain $D$, and $\Phi$ be a  set of linear measurements  such that
 the
map $T_\Phi$ in  \eqref{unitary.injective} is injective.
Then
${\mathcal S}$ is  phase retrieval  if and only if
$\Phi$ has the complement property  \eqref{complementary.def}.
\end{cor}

\subsection{Phase retrieval of projections}\label{projective.subsection}
Let   $\tilde {\mathcal H}$ be a linear subspace of the  Hilbert space ${\mathcal{H}}$,
and $P_{\tilde {\mathcal H}}$ be
 the projection from ${\mathcal{H}}$ onto ${\tilde {\mathcal H}}$.
In this section, we establish the following result on phase retrieval of
the projection of % vector-valued
vector-valued functions onto $\tilde {\mathcal H}$,
 see Section \ref{subspacePR.thm.pfsection} for the proof.

\begin{thm}\label{subspacePR.thm}
Let ${\mathcal{H}}$ and $\tilde {\mathcal H}$ be a real separable Hilbert space and  its linear subspace respectively,
 ${\mathcal S}$ be a unitary invariant space of ${\mathcal H}$-valued functions
 on the domain $D$, and  let $\Phi$ be a unitary invariant set of linear measurements satisfying Assumption  \ref{unitary.assumption}.
   If $f\in {\mathcal S}$ is phase retrieval in ${\mathcal S}$, then its projection  $P_{\tilde {\mathcal H}} f$
onto ${\tilde {\mathcal H}} $ is  phase  retrieval in $P_{\tilde {\mathcal H}} {\mathcal S}$, i.e.,
\begin{equation}\label{subspacePR.cor.eq1}
{\mathcal M}_{P_{\tilde {\mathcal H}} f, \Phi}=\big\{\tilde U P_{\tilde {\mathcal H}} f,   \
 \tilde U \in {\mathcal U}({\tilde {\mathcal H}})\big\}.
\end{equation}
\vskip .1in
 %and $\phi$ and $P_{\mathbf G}$ are commutable for all $\phi \in \Phi$.

%{\color{red} The space $V$ is invariant under the projection $P_G$. Moreover $P_GV\subset V$. Then,
%the orthogonal transformation in \ref{subspacePR.cor.eq1} should be in $O({\bf H})$?  }
% and
%\begin{equation}\label{subspacePR.cor.eq2}
%{\mathcal M}_{P_{\mathbf G} f}=\{g\in P_{\mathbf G}V:  \ \|\phi(g)\|=\|\phi(P_{\mathbf G}f)\| \ {\rm for \ all} \ \phi\in W\},
%\end{equation}
%where $W$ is defined as in \eqref{dualsapce.def}.
%%where $W$ is a subset of the dual space $(V^J)'$ such that  any function $f\in V^J$ can be uniquely
% determined by its observations $\{\phi(f)\}_{\phi\in W}$.
%% Assume that ${\mathcal H}_0$ be  a closed linear subspacece of $H$ and $f\in V$, denote the projection
%of $f(t), t\in D$ onto $H_1$ by $P_{H_0}f(t)$  and $V_0$ be a linear space of ${\mathcal H}_1$-valued spatial functions on a domain $D$.
%For any $f\in V$, define
%%\begin{equation} {\mathcal M}^0_f=\{g\in V_1: \ \|g(x)\|=\|P_{{\mathcal H}_0}f(x)\|, x\in D\}.\end{equation}
\end{thm}

For  a nonzero vector  ${\bf e}\in {\mathcal H}$, set
$ {\mathcal S}_{\bf e}=\{\langle f,  {\bf e}\rangle, \ f\in {\mathcal S}\}$.
Then ${\mathcal S}_{\bf e}$ is a linear space of real functions on the domain $D$. Applying Corollary \ref{scalepr.cor}  and
Theorem \ref{subspacePR.thm} with $\tilde {\mathcal H}$ replaced by the one-dimensional space spanned by ${\bf e}$, we have
the following corollary about phase  retrieval of real functions in the linear space ${\mathcal S}_{\bf e}$.

\begin{corollary}
Let ${\mathcal{H}},
{\mathcal S}, \Phi$ be as in Theorem  \ref{subspacePR.thm}. If
  $f\in {\mathcal S}$ is phase retrieval in ${\mathcal S}$, then
 for any ${\bf e}\in {\mathcal H}$,  the real function
 $\langle f, {\bf e}\rangle$  is phase  retrieval in ${\mathcal S}_{\bf e}$, or equivalently there do not exist $g, h\in {\mathcal S}\setminus (\!\span\{\bf e\})^\bot$ such that
 \begin{equation}
 \langle f, {\bf e}\rangle=\langle g+h, {\bf e}\rangle\ {\rm and} \ \langle \phi(g), {\bf e}\rangle \langle \phi(h), {\bf e}\rangle=0 \ {\rm for \ all} \ \phi\in \Phi.
 \end{equation}
\end{corollary}

\smallskip

\subsection{Phase retrieval of vector fields on  graphs}\label{vector.subsection}
 Let ${\mathcal G}=(V, E)$ be a simple graph and ${\bf f}=({\bf f}_i)_{i\in V}$ be a $d$-dimensional vector field on ${\mathcal G}$, where
${\bf f}_i\in \Rd, i\in V$. First we consider the case that  ${\mathcal G}$ is  a complete graph. In this case, we obtain from \eqref{vectorfield.thm.pfeq1} and \eqref{vectorfield.thm.pfeq2} that
\begin{equation}\label{completegraph.eq00}
\langle {\bf g}_i, {\bf g}_j\rangle= \langle {\bf f}_i, {\bf f}_j\rangle\ {\rm \ for\ all} \ i, j\in V.
\end{equation}
Applying \eqref{completegraph.eq00} and  following the argument used in Remark \ref{assump.remark}, we can find  an orthogonal matrix $U$ of size $d\times d$ such that
$${\bf g}_i= U {\bf f}_i, \  i\in V.$$
This leads to the following result.

\begin{thm}\label{vectorfieldcompletegraph.thm}
Let ${\mathcal G}=(V, E)$ be a complete graph. Then
 a vector field  ${\bf f}=({\bf f}_i)_{i\in V}$ can be reconstructed,  up to an  orthogonal matrix,  from its  absolute magnitudes
$\|{\bf f}_i\|, i\in V$ and relative magnitudes $\|{\bf f}_i-{\bf f}_j\|, i,j\in V$.
\end{thm}

From the above theorem, we may conclude that a velocity field on a complete graph is determined, up to an orthogonal matrix,
 from its absolute speed  at vertices and relative speed  between vertices.
We remark that the conclusion in  Theorem \ref{vectorfieldcompletegraph.thm} does not hold for an arbitrary graph.
Let $C_n$ be the circulant graph  with
 nodes labeled $0,1,\ldots ,n-1$  and  each node $i$  adjacent to nodes $i\pm 1 \mod n$.
 One may verify that the vector fields ${\bf h}=({\bf h}_i)_{0\le i\le n-1}$ on the circulant graph  $C_n$ of even order
 have the same  magnitudes $\|{\bf h}_i\|=1$ on all vertices $0\le i\le n-1$  and relative magnitudes
 $\|{\bf h}_i-{\bf h}_{i\pm 1}\|=\sqrt{2}$
at all neighboring vertices, where
 \begin{equation*}\label{counterex.vec}
 {\bf h}_{i}=\left\{\begin{array}{ll}  \pm (1, 0)^T  & {\rm if} \ i \ {\rm is \ even} \\
 \pm (0, 1)^T   & {\rm if } \  i \ {\rm is \ odd}.
 \end{array}
 \right.
 \end{equation*}
 % {\color{red} see the left in Figure \ref{simplexgraph.fig}. }

For a simple graph ${\mathcal G}=(V, E)$  and a $d$-dimensional vector field ${\bf f}=({\bf f}_i)_{i\in V}$ on ${\mathcal G}$,
we define  % {\color{red} a d-simplex is }
\begin{equation}\triangle({\bf f}, {\mathcal G}_c)=\Big
\{\sum_{i\in V_c}t_i{\bf f}_i, \ \sum_{i\in V_c} t_i=1\ {\rm and}\ 0\le t_i\le 1\ {\rm for \ all}\   i\in V_c\Big\},
\end{equation}
where
${\mathcal G}_c=(V_c, E_c)$ is a complete subgraph  of order $d+1$.
 Let $V_{\bf f}$ be the set of all complete subgraphs
 ${\mathcal G}_c$
 of order $d+1$ such that $\triangle({\bf f}, {\mathcal G}_c)$ is a  $d$-simplex in $\Rd$, and  $E_{\bf f}$ be
 the set of all pairs  of
 complete subgraphs   ${\mathcal G}_c, \tilde {\mathcal G}_c \in V_{\bf f}$
 of order $d+1$  such that
 $\triangle({\bf f}, {\mathcal G}_c\cap  \tilde {\mathcal G}_c)$ is a $(d-1)$-simplex
 and  the hyperplane containing $\triangle({\bf f}, {\mathcal G}_c\cap  \tilde {\mathcal G}_c)$ does not include the origin.
We call the graph ${\mathcal G}_{\bf f}=(V_{\bf f}, E_{\bf f})$
 with the vertex  set $V_{\bf f}$ and  edge set $ E_{\bf f}$  defined above as
the {\em $d$-simplex graph} associated with the vector field ${\bf f}$, see Figure \ref{simplexgraph.fig}.

%{\color{red} \large \bf Add a figure here}
\begin{figure}[h]% [h]  %[t]
\begin{center}
\includegraphics[width=118mm, height=52mm]{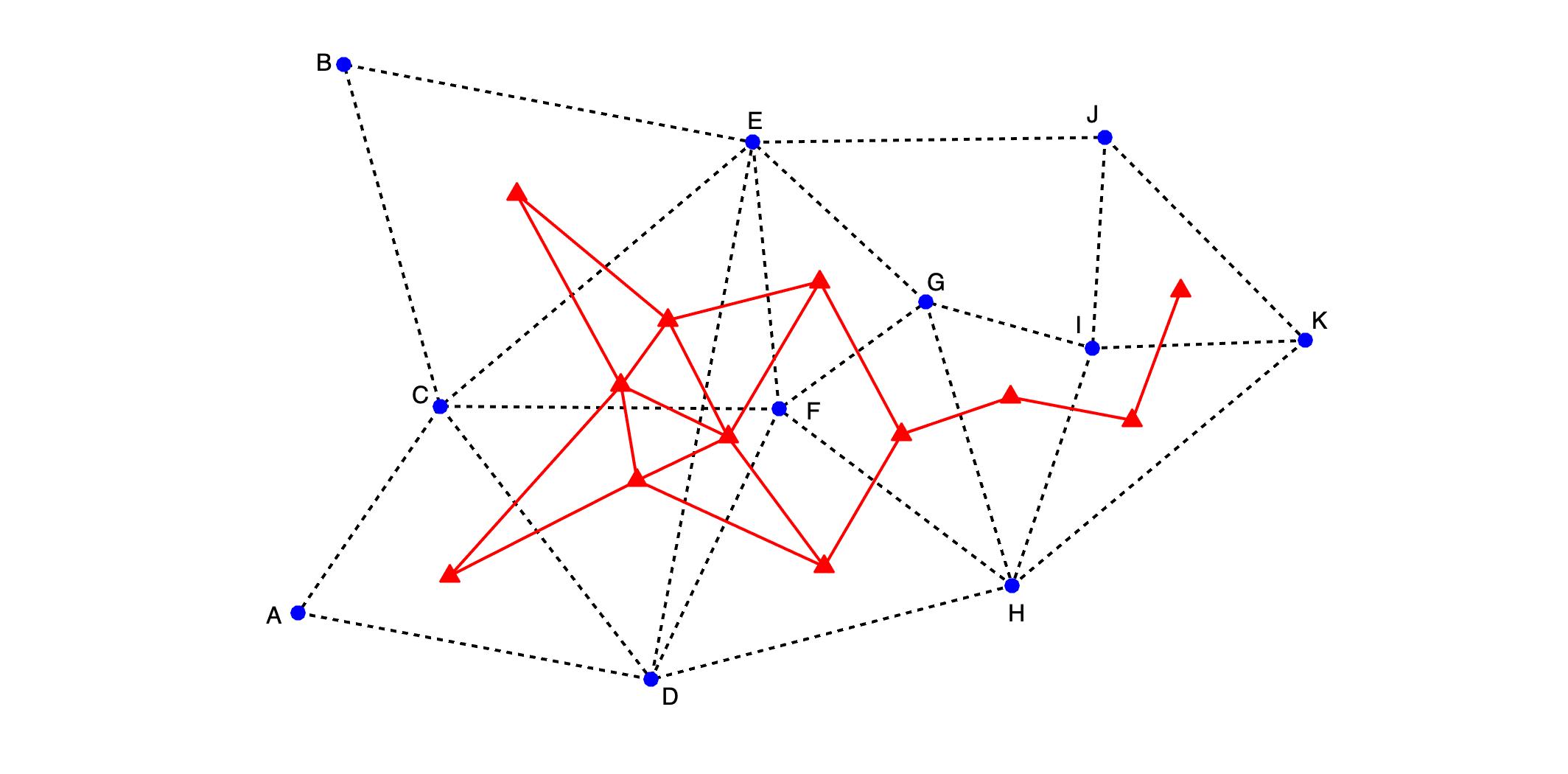}
\caption{%Plotted on the left in red is the circulant graph.
%Plotted on the left in blue is the one satisfying \eqref{vectorfield.thm.pfeq1} and \eqref{vectorfield.thm.pfeq1},
% but not the same as the graph in red, even up to an orthogonal transformation.
 Plotted are  the  graph ${\mathcal G}$ with vertices marked by  blue dots from A to K and edges in black dashed  lines,
and the $2$-simplex graph ${\mathcal G}_{\bf f}$ with vertices  $\{\triangle ACD, \triangle BCE, \triangle CDE, \triangle CDF,
\triangle CEF, \triangle DEF, \triangle DFH,$
$\triangle EFG, \triangle FGH, \triangle GHI, \triangle HIK, \triangle IJK\}$ marked  by the red filled triangle located around the centroid
%at the center of each 2-simplex (triangle)
 and edges between  2-simplices  in red   solid lines, where we assume that the hyperplane containing two vectors of the vector field ${\bf f}$
 located at edges of the graph ${\mathcal G}$, except
 at the edge between vertices $C$ and $F$, does not pass through the origin.
}\label{simplexgraph.fig}
%\vskip-0.9mm
\end{center}
\end{figure}

  In the following theorem, see  Section \ref{vectorfield.thm.section} for the proof,  we show that
a   vector  field  ${\bf f}$ on a simple graph  ${\mathcal G}$ is determined, up to  an orthogonal matrix,
 from its absolute magnitudes at vertices and relative  magnitudes between  neighboring vertices if
 the $d$-simplex graph  ${\mathcal G}_{\bf f}$ is connected.

\begin{thm}\label{vectorfield.thm}
Let ${\mathcal G}=(V, E)$ be a simple finite graph, ${\bf f}=({\bf f}_i)_{i\in V}$ be a vector field on the graph   ${\mathcal G}$
with its
$d$-simplex graph denoted by  ${\mathcal G}_{\bf f}$.
If the $d$-simplex graph ${\mathcal G}_{\bf f}$ is connected and  for any vertex $i\in V$ there exists a
complete subgraph
 ${\mathcal G}_c=(V_c, E_c)$
 of order $d+1$ such that   $i\in V_c$ and $\triangle({\bf f}, {\mathcal G}_c)$ is a $d$-simplex,
  then  the vector field ${\bf f}$ is  determined, up to an orthogonal matrix,
 from its  absolute magnitudes $\|{\bf f}_i\|, i\in V$ at vertices and relative magnitudes $\|{\bf f}_i-{\bf f}_j\|, (i,j)\in E$ of neighboring vertices.
\end{thm}

The $d$-simplex graph ${\mathcal G}_{\bf f}$ is defined when the vector field ${\bf f}$ is given.
%{\color{red} One may verify that the 2-simplex graphs
%of the vector fields ${\bf h}$ with \eqref{counterex.vec} on the circulant graph is not connected.}
In the following remark, we show that
the $d$-simplex graph ${\mathcal G}_{\bf f}$
can also be constructed from the available  absolute magnitudes $\|{\bf f}_i\|, i\in V$ %at vertices
 and relative magnitudes $\|{\bf f}_i-{\bf f}_j\|, (i,j)\in E$. %where $V$ and $E$ are the vertex set and the edge set % of neighboring vertices
%of the graph ${\mathcal G}$.

\begin{remark}\label{simplegraph.rem}{\rm
For a complete subgraph ${\mathcal G}_c=(V_c, E_c)$
of order $d+1$,
one may verify that the  convex set $\triangle({\bf f}, {\mathcal G}_c)$ is a $d$-simplex if and only if
${\bf f}_i, i\in V_c$, are  {\em affinely independent} in the sense that
\begin{equation}\label{affineindependent.def}
\sum_{i\in V_c} c_i {\bf f}_i=0{\rm\  and}\ \sum_{i\in V_c} c_i=0 \ {\rm implies} \ c_i=0 \ {\rm for \ all} \ i\in V_c,
\end{equation}
 or equivalently  vectors
${\bf f}_i-{\bf f}_{i_0}\in \Rd, i\in V_c\backslash \{i_0\}$,
are linearly independent for some $i_0\in V_c$, or equivalently
the Gramian matrix  $(\langle {\bf f}_i, {\bf f}_j\rangle)_{i,j\in V_c}$ is strictly positive definite on the hyperplane
$$\Big\{(c_i)_{i\in V_c}, \  \sum_{i\in V_c} c_i=0\Big\}\subset \R^{d+1}.$$
Observe that entries in the Gramian matrix are given by
\begin{equation}\label{gram2}
\langle {\bf f}_i, {\bf f}_j\rangle=\frac{1}{2}\big( \| {\bf f}_i\|^2+  \|{\bf f}_j\|^2- \| {\bf f}_i-{\bf f}_j\|^2\big), i,j\in V_c.
\end{equation}
Then the vertex set $V_{\bf f}$ of the $d$-simplex graph ${\mathcal G}_{\bf f}$ can be constructed, from
the  absolute magnitudes $\|{\bf f}_i\|, i\in V$   and relative magnitudes $\|{\bf f}_i-{\bf f}_j\|, (i,j)\in E$. % of neighboring vertices
%of the graph ${\mathcal G}$.

 Given a pair of
 complete subgraphs
 ${\mathcal G}_c=(V_c, E_c)$ and $\tilde {\mathcal G}_c=(\tilde V_c, \tilde E_c)$
 in the vertex set  $V_{\bf f}$, % of the $d$-simplex graph  ${\mathcal G}_{\bf f}$,
 one may verify that
   $\triangle({\bf f}, {\mathcal G}_c\cap  \tilde {\mathcal G}_c)$ is a $(d-1)$-simplex
 if and only if
 $V_c\cap \tilde V_c$ has cardinality $d$, and also
   that the hyperplane containing $\triangle({\bf f}, {\mathcal G}_c\cap  \tilde {\mathcal G}_c)$ does not include the origin
   if and only if
   ${\bf f}_i, i\in V_c\cap \tilde V_c$, are linearly independent or equivalently
   the Gramian matrix $(\langle {\bf f}_i, {\bf f}_j\rangle)_{i,j\in V_c\cap \tilde V_c}$ is  strictly positive definite.
   This together with \eqref{gram2} implies that
   edges in  the $d$-simplex graph  ${\mathcal G}_{\bf f}$
   are determined from
the available absolute magnitudes $\|{\bf f}_i\|, i\in V$  and relative magnitudes $\|{\bf f}_i-{\bf f}_j\|, (i,j)\in E$. % of neighboring vertices of the graph ${\mathcal G}$.
 }
 \end{remark}

\section{Affine phase retrieval of vector-valued functions}\label{affine.pr.sec}

% In holography, the task is to recover
%the objective wave ${\bf Y}({\bf x})$ at every location ${\bf x}$ from the hologram $I({\bf x})=|\textbf{Y}(\textbf{x})+\textbf{R}(\textbf{x})|$
%recorded by the CCD camera
%for the  interference between the   objective wave ${\bf Y}({\bf x})$ and  a given reference wave $\textbf{R}(\textbf{x})$  \cite{LBCMDU2003}.
%In some data separation problems, we need separate a true signal ${\bf x}$ from the  known background/noise signal ${\bf y}$ when
%quadratic measurements $|{\bf A}{\bf x}+{\bf B}{\bf y}|$ of their mixture are available only, where
%${\bf A}$ and ${\bf B}$ are mixing matrices
% \cite{DH2001,DK2013,CMVG2010, LLS2013}. In phaseless sampling and reconstruction of temporal signals, one want to recover a signal
% ${\bf x}$ at a time period
% for its phaseless measurements $|{\bf A}{\bf x}+{\bf B}{\bf y}|$, where ${\bf y}$ is a known temporal signal
%on the prior time period \cite{CCSW2016}.

Let $\mathcal H$ be a separable Hilbert space and ${\mathcal S}$ be a linear space of $\mathcal H$-valued functions on a domain  $D$,   $\Phi$ be a family of linear measurements,
and  let  $B_{\Phi, N}=\{b_{\phi, i} , \phi\in\Phi, 1\le i\le N\}$  be a family of   reference vectors in  $\H$ for each linear measurement.
For any $f\in {\mathcal S}$, we define %Define the map $M: V\rightarrow \bf H$ by
\begin{equation*}
{\mathcal A}_{f}=\{g\in {\mathcal S},  \ \|\phi(f)+b_{\phi, i}\|=\|\phi(g)+b_{\phi, i}\|, \ \phi\in \Phi, \ 1\le i\le N\}.
\end{equation*}
We say that $f\in {\mathcal S}$ is  {\em affine phase retrieval} in ${\mathcal S}$ if
 $f$ is uniquely determined from its phaseless affine measurements $\|\phi (f)+b_{\phi,i}\|, \phi\in \Phi, 1\le i\le N$, i.e.,
\begin{equation}\label{affinepr.def0} {\mathcal A}_{f}=\{f\},\end{equation}
 and
the whole space ${\mathcal S}$ is affine phase retrieval if every vector-valued function in ${\mathcal S}$ is affine phase retrieval,
see \cite{GSXW, HX18, KST95, CLS19} for affine phase retrieval of complex/real functions (vectors).
In this section, %Section \ref{apr.section}, % this section, % Section \ref{affine.pr.sec},
we consider the affine phase retrieval of  vector-valued  functions in a linear space, see Theorems \ref{affinepf.thm} and \ref{affinepfmultiple.thm}.

%  Let  $R(T_\Phi)$ be the range space of the  linear map  $T_\Phi$ in \eqref{unitary.injective}.
% In Section \ref{referencevectors.section}, we observe that if all reference vectors $(b_{\phi, i})_{\phi\in \Phi}, 1\le i\le N$,
%are chosen from the range space $R(T_\Phi)$, then
%  any linear space ${\mathcal S}$  of vector-valued functions with dimension larger than $n$ is not affine phase retrieval, see
%  Corollary \ref{affinedimension.cr}.  This indicates that
%  for the affine phase  retrieval of vector-valued functions in a linear space  ${\mathcal S}$,
%we should select some group of reference vectors not in the range space $R(T_\Phi)$.
% %being the same as the linear measurement of some vector-valued functions in ${\mathcal S}$.

%\subsection{Affine phase retrievability}\label{apr.section}
 For all $f, g\in {\mathcal S}$ having the same  phaseless affine measurements, i.e.,
 $$\|\phi(f)+b_{\phi, i} \|=\|\phi(g)+b_{\phi, i}\| \ \ {\rm for \ all} \  \phi\in \Phi\ {\rm and} \ 1\le i\le N,$$ we have
\begin{eqnarray*}
\langle \phi(f-g), \phi(f+g)+2b_{\phi, i} \rangle=0,  \phi\in \Phi, 1\le i\le N.
\end{eqnarray*}
This leads to the following characterization to   affine phase retrieval of a linear space.
%{\color{red} Keep uniform as the other theorems}
\begin{thm}\label{affinepf.thm}  Let ${\mathcal S}$ be a linear space of ${\mathcal H}$-valued functions on a domain  $D$,   $\Phi$ be a family of
linear measurements on ${\mathcal S}$, and let  $B_{\Phi, N}=\{b_{\phi, i} , \phi\in\Phi, 1\le i\le N\}$  be a family of reference vectors in  $\H$. Then
 the  linear space   ${\mathcal S}$ % of vector-valued functions on a domain  $D$
  is  affine phase retrieval if and only if
for any $u\in {\mathcal S}$,
the linear  map  $T_u$  % on ${\mathcal S}$
defined  by
 \begin{equation}\label{Tuaffine.def}
 T_u: {\mathcal S}\ni v\longmapsto (\langle \phi(v), \phi(u)+b_{\phi, i} \rangle)_{\phi\in \Phi, 1\le i\le N}\end{equation}
is injective.

%\item[(iii)] For any $u\in {\mathcal S}$, the  map  $T_u: {\mathcal S}\ni v\longmapsto
% (\langle \phi(v), \phi(u)+a_\phi \rangle)_{\phi\in \Phi}$ is injective.
%\end{itemize}
\end{thm}

%{\color{blue} \bf Delete the following}{\color{red}  Applying the  characterization in Theorem  \ref{affinepf.thm} to the  real range space
%\begin{equation}\label{realrange.def}
%R({\bf A})= \Big \{{\bf A}{\bf x}, \ {\bf x}\in  \R^n\Big\}
%\end{equation}
%of a real matrix  ${\bf A}$, we reach a  conclusion  in \cite[Theorem 2.1]{GSXW}.
%
%\begin{corollary} Let the real matrix ${\bf A}$ have full rank $n$. Then
%any function ${\bf y}=(y_1, \ldots, y_m)^T\in R({\bf A})$ is
%determined from its affine measurements
%$| y_k+b_{k, i}|, 1\le k\le m, 1\le i\le n$, if and only if for any ${\bf u}=(u_1, \ldots, u_m)^T\in R({\bf A})$
%there does not exist a nonzero vector ${\bf v}=(v_1, \ldots, v_m)^T\in R({\bf A})$  such that
%$${v}_k ({ u}_k+b_{k, i}){\color{red}= 0}  \ {\rm for \ all} \  1\le k\le m\ {\rm and}\ 1\le i\le n.$$
% \end{corollary}}

 Applying %the  characterization in
 Theorem  \ref{affinepf.thm}
 to  the  complex range space  $R_{\C}({\bf A})$ in \eqref{complexrange.def}, we obtain a  conclusion given in \cite[Theorem 3.1] {GSXW}.

\begin{corollary} Let the real matrix ${\bf A}$ have full rank $n$. Then
any function ${\bf y}=(y_1, \ldots, y_m)^T\in R_{\C}({\bf A})$ is
determined from its affine measurements
$| y_k+b_{k, i}|, 1\le k\le m, 1\le i\le N$, if and only if for any ${\bf u}=(u_1, \ldots, u_m)^T\in R_{\C}({\bf A})$
there does not exist a nonzero vector ${\bf v}=(v_1, \ldots, v_m)^T\in R_{\C}({\bf A})$  such that
$$\Re{{v}_k^\ast ({ u}_k+b_{k, i})}= 0 \ {\rm for \ all} \  1\le k\le m \ {\rm and} \ 1\le i\le N.$$
 \end{corollary}
 We remark that the above characterization also holds for the real linear subspace of $R_{\C}({\bf A})$ in \eqref{complexrange.def},
 %where $\bf A$ have rank $n$,
  cf. \cite[Theorem 2.1]{GSXW}.

%{\color{red} By  Theorem \ref{affinepf.thm}, we can only verify the  injectivity of $T_u$ in \eqref{Tuaffine.def} for all $u\in {\mathcal S}$ for the affine phase retrieval of a linear space ${\mathcal S}$.  }
By  Theorem \ref{affinepf.thm}, the  verification of  affine phase  retrieval of a linear space ${\mathcal S}$  reduces to checking
the injectivity of  $T_u$ in \eqref{Tuaffine.def} for all $u\in {\mathcal S}$.
In the following theorem, see Section \ref{affinepfmultiple.thm.pfsection} for the proof, we  provide a simpler characterization to
  affine phase retrieval
when there are multiple reference vectors for each linear measurement, i.e., $N\ge 2$.

\begin{thm}\label{affinepfmultiple.thm}
Let $N\ge 2$, ${\mathcal S}$ be a linear space of ${\mathcal H}$-valued functions on a domain  $D$,   $\Phi$ be a family of
linear measurements on ${\mathcal S}$, and let  $B_{\Phi, N}=\{b_{\phi, i} , \phi\in\Phi, 1\le i\le N\}$  be a family of reference vectors in  $\H$.
%For $f\in {\mathcal S}$, set
%\begin{equation*}
%N_f=\big\{g\in {\mathcal S},\ \|\phi(f)+b_{\phi, i}\|=\|\phi(g)+b_{\phi, i}\| \  {\rm for\ all} \ \phi\in \Phi \ {\rm and} \ 1\le i\le n \big\}.
%\end{equation*}
Then
a sufficient condition for  affine phase  retrieval of the linear space  ${\mathcal S}$ %i.e.,
%$${\mathcal N}_f=\{f\} \ {\rm for \ all} \  f\in {\mathcal S},$$
is injectivity of the linear map
\begin{equation}\label{Tijaffine.def}
 T: {\mathcal S}\ni v\longmapsto (\langle \phi(v), b_{\phi, i}-b_{\phi, j} \rangle)_{\phi\in \Phi, 1\le i, j\le N}.\end{equation}
The injectivity of the map $T$ is  necessary for
affine phase retrieval of the linear space  ${\mathcal S}$
if one group of  the reference vectors  is the linear measurement of a function $f_0\in {\mathcal S}$, i.e.,
 there exists $1\le i_0\le N$ such that
 \begin{equation}\label{bphiio.def}
 b_{\phi, i_0}= \phi(f_0),\ \phi\in \Phi. \end{equation}
 \end{thm}

%there does not exist  $0\neq v\in {\mathcal S}$ such that
%\begin{equation}\label{affinedecomp.1}
%\langle\phi(v),\phi(u)+b_\phi\rangle= 0 \ {\rm for \ all} \ \phi\in \Phi.
%\end{equation}
% For any $u\in {\mathcal S}$,  define a linear  map  $T_u$ on ${\mathcal S}$ by
% \begin{equation}\label{Tuaffine.def}
% T_u: {\mathcal S}\ni v\longmapsto (\langle \phi(v), \phi(u)+b_\phi \rangle)_{\phi\in \Phi}.\end{equation}
% Then by \eqref{affinepf.thm},
% the injectivity of the above linear map $T_u$ for every $u\in {\mathcal S}$ is another characterization to
% affine phase retrievability of the linear space ${\mathcal S}$.
%% We remark that a necessary condition for  the injectivity of the above

A necessary condition for
affine phase retrieval of the linear space  ${\mathcal S}$ is that
the map $T_\Phi$ in \eqref{unitary.injective} is injective.
The above necessary condition is also sufficient when the reference vectors
 $B_{\Phi, N}=\{b_{\phi, i} , \phi\in\Phi, 1\le i\le N\}$
 satisfies
\begin{equation}\label{HBphin}
{\mathcal H}={\rm span} \{b_{\phi, i}-b_{\phi, j}, 1\le i,j\le N\} \ {\rm \ for \ all} \ \phi\in \Phi,\end{equation}
 as the map $T$
in \eqref{Tijaffine.def} is injective under the above assumption.
Therefore applying Theorem \ref{affinepfmultiple.thm}, we have
the  following equivalence between   affine phase  retrieval  of the linear space ${\mathcal S}$
and  injectivity of the map $T_\Phi$, cf. \cite{GSXW} for the scalar setting.
%, {\color{red} which is established for  the real range space in
% \eqref{realrange.def}.  }
%a vector-valued generalization of Corollary
%\eqref{scalar.affine.cor}.

\begin{cor}\label{vector.affine.cor}
Let $N\ge 2$, ${\mathcal S}$ be a linear space of vector-valued functions on a domain  $D$,   $\Phi$ be a family of linear
 measurements on ${\mathcal S}$, and let  $B_{\Phi, N}=\{b_{\phi, i} , \phi\in\Phi, 1\le i\le N\}$  be a family of reference vectors in
  $\H$ satisfying \eqref{HBphin}.
Then ${\mathcal S}$ is affine phase retrieval  if and only if the map $T_\Phi$ in \eqref{unitary.injective} is injective.
\end{cor}

In the scalar setting, i.e., ${\mathcal H}=\R$,  the requirement \eqref{HBphin} is satisfied only if
$ F$ is an empty set, where
$ F=\{\phi\in \Phi, \  \sum_{i,j=1}^n |b_{\phi, i}-b_{\phi, j}|^2=0\}$.
Applying Theorem \ref{affinepfmultiple.thm}, we have the following slight generalization of Corollary
\ref{vector.affine.cor} in the scalar setting.

%the inner product between $\phi(v)$ and $b_{\phi, i}-b_{\phi, j} \rangle$ is zero if and only if either $\phi(v)=0$ or
%$b_{\phi, i}-b_{\phi, j} =0$. Then applying Theorem \ref{affinepfmultiple.thm}, we have
%

\begin{cor}\label{scalar.affine.cor}
Let  $N\ge 2$,  ${\mathcal S}$ be a linear space of scalar-valued functions on a domain  $D$,   $\Phi$ be a family of
 linear measurements on ${\mathcal S}$, and let  $B_{\Phi, N}=\{b_{\phi, i} , \phi\in\Phi, 1\le i\le N\}$  be a family of reference real numbers.
Then ${\mathcal S}$ is affine phase retrieval  if there does not exist a nonzero function $f$ such that
$\phi(f)=0$ for all $\phi\not \in F$.
%, where
% $U=\{\phi\in \Phi: \sum_{i,j=1}^n |b_{\phi, i}-b_{\phi, j}|=0\}$.
\end{cor}

\section{Proofs}\label{proofs.section}

 In this section, we collect the  proofs of Theorems \ref{hilbertpr.thm}, \ref{subspacePR.thm}, \ref{vectorfield.thm},
  \ref{complexpr.thm}, \ref{necess.charac}, \ref{quaternionpr.thm} and \ref{affinepfmultiple.thm},
 and Propositions  \ref{mtwopr.prop} and \ref{complex.sis.pr}.
%  and  \ref{affinedimension.pr}.

\subsection{Proof of Theorem \ref{hilbertpr.thm}}\label{hilbertpr.thm.section}

First the necessity.   Suppose, on the contrary, that there exist
$u, v\in {\mathcal S}$ such that \eqref{hilbertpr.thm.eq3}, \eqref{hilbertpr.thm.eq4} and
\eqref{hilbertpr.thm.eq5} hold.
Set
$ g=u-v$.
From %and \eqref{hilbertpr.thm.proof.eq1}
\eqref{hilbertpr.thm.eq3} and \eqref{hilbertpr.thm.eq4}, we obtain
\begin{eqnarray*}  \label{hilbertpr.thm.proof.eq2}
\|\phi(g)\|^2  %\hskip-0.1in&=\hskip-0.05in&\hskip-0.05in  \langle \phi(u)-\phi(v), \phi(u)-\phi(v)\rangle\nonumber\\
\hskip-0.1in&=\hskip-0.05in&\hskip-0.05in
\|\phi(u)\|^2+\|\phi(v)\|^2-2\langle  \phi(u), \phi(v)\rangle\nonumber\\
\hskip-0.1in&=\hskip-0.05in&\hskip-0.05in \|\phi(u)\|^2+\|\phi(v)\|^2+2\langle \phi(u), \phi(v)\rangle
%\hskip-0.1in&=\hskip-0.05in&\hskip-0.05in\langle u(x)+v(x), u(x)+v(x)\rangle
=\|\phi(f)\|^2
\end{eqnarray*}
%For any $x\in D$, we obtain from
%\eqref{hilbertpr.thm.eq3}, \eqref{hilbertpr.thm.eq4} and \eqref{hilbertpr.thm.proof.eq1} that
%\begin{eqnarray}  \label{hilbertpr.thm.proof.eq2}
%\|g(x)\|^2\hskip-0.1in&=\hskip-0.05in&\hskip-0.05in  %\langle u(x)-v(x), u(x)-v(x)\rangle=
%\|u(x)\|^2+\|v(x)\|^2-2\langle u(x), v(x)\rangle\nonumber\\
%\hskip-0.1in&=\hskip-0.05in&\hskip-0.05in \|u(x)\|^2+\|v(x)\|^2+2\langle u(x), v(x)\rangle
%%\hskip-0.1in&=\hskip-0.05in&\hskip-0.05in\langle u(x)+v(x), u(x)+v(x)\rangle
%=\|f(x)\|^2,
%\end{eqnarray}
for all $\phi\in \Phi$. This proves that
\begin{equation} \label{hilbertpr.thm.proof.eq3} g\in {\mathcal M}_{f, \Phi}.\end{equation}
By \eqref{hilbertpr.thm.eq5}, we have %\eqref{hilbertpr.thm.proof.eq1},  \eqref{hilbertpr.thm.eq3} and
$$
\langle \phi_0(f), \phi_1(f)\rangle- \langle \phi_0(g), \phi_1(g)\rangle
= 2 (\langle \phi_0(u), \phi_1(v)\rangle + \langle \phi_1(u), \phi_0(v)\rangle)\ne 0
$$
%  %\nonumber\\
%\hskip-0.1in&=\hskip-0.05in&\hskip-0.05in \langle u(x_0), u(y_0)\rangle+\langle v(x_0), u(y_0)\rangle+\langle u(x_0), v(y_0)\rangle+\langle v(x_0), v(y_0)\rangle\nonumber\\
%- \langle g(x), g(y)\rangle
%\hskip-0.1in&\neq\hskip-0.05in&\hskip-0.05in \langle u(x_0), u(y_0)\rangle-\langle v(x_0), u(y_0)\rangle-\langle u(x_0), v(y_0)\rangle+\langle v(x_0), v(y_0)\rangle\nonumber\\
%\hskip-0.1in&=\hskip-0.05in&\hskip-0.05in \langle g(x_0), g(y_0)\rangle
%\end{eqnarray*}
for some $\phi_0, \phi_1\in \Phi$. Hence
%{\color{red} Assumption: $\phi$ and $T\in O(H)$ commutates.}
%\begin{equation}  \label{hilbertpr.thm.proof.eq5}
$g\ne Uf$ for all  $U\in {\mathcal U}({\mathcal{H}})$. This together with
\eqref{hilbertpr.thm.proof.eq3} contradicts to the  phase retrieval of  function $f$.
% is roof.eq5}
% contradicts to the phase retrievability of $f$.% is phase retrieval. % the assumption \eqref{hilbertpr.thm.eq2}.

Now the sufficiency.  Suppose, on the contrary, that there exists
$g\in {\mathcal M}_{f, \Phi}$ such that
$g\ne Uf$ for all $U\in {\mathcal U}(\mathcal{H})$.
Then there exist $\phi_0, \phi_1\in\Phi$ by  Assumption \ref{unitary.assumption}  such that
% $x_0, y_0\in D$ such that
\begin{eqnarray}  \label {orthogonal.lem.pf.eq6}
\langle \phi_0(g), \phi_1(g)\rangle \neq \langle \phi_0(f), \phi_1(f)\rangle.
\end{eqnarray}
Set
$u= (f+g)/2$ and $v=(f-g)/2\in {\mathcal S}$.
 One may verify that
$f= u+v$,
 \begin{equation*}
\langle \phi(u), \phi(v)\rangle= \frac{1}{4} \big(\langle \phi(f), \phi(f)\rangle - \langle \phi(g), \phi(g)\rangle \big)=0
\end{equation*}
 for all $\phi\in \Phi$ by  the assumption that $g\in {\mathcal M}_{f, \Phi}$,
 and
%\begin{eqnarray}\label{orth.decom.1}
%\langle u, v\rangle=\frac{1}{4}\langle f+g, f-g\rangle=\frac{1}{4}[\langle f, f \rangle-\langle g, g \rangle-\langle f, g \rangle+\langle g, f \rangle ]=0.
%\end{eqnarray}
%\begin{eqnarray}  \label {orthogonal.lem.pf.eq6}
%\langle f(x_0), f(y_0)\rangle \neq \langle g(x_0), g(y_0)\rangle.
%\end{eqnarray}
\begin{equation*}  \label {orthogonal.lem.pf.eq7}  %\label{orth.decom.2}
\langle \phi_0(u), \phi_1(v) \rangle+\langle \phi_0(v), \phi_1(u) \rangle
=\frac{1}{2}\big(\langle \phi_0(f), \phi_1(f)\rangle-\langle \phi_0(g), \phi_1(g)\rangle\big)\ne 0  %, \ \lambda, \mu\in \Lambda.
\end{equation*}
by  \eqref{orthogonal.lem.pf.eq6}.
%\begin{equation}  \label {orthogonal.lem.pf.eq7}  %\label{orth.decom.2}
%\langle u(x), v(y) \rangle+\langle u(y), v(x) \rangle
%%\nonumber\\
%%\hskip-0.1in&=\hskip-0.05in&\hskip-0.05in\frac{1}{4}\langle (f+g)(x_0), (f-g)(x_0)\rangle+\frac{1}{4}\langle (f+g)(y_0), (f-g)(x_0)\rangle\nonumber\\
%%\hskip-0.1in&=\hskip-0.05in&\hskip-0.05in
%=\frac{1}{2}\big(\langle f(x), f(y)\rangle-\langle g(x), g(y)\rangle\big), \ x, y\in D.
%\end{equation}
%Combining  \eqref{orthogonal.lem.pf.eq6} and \eqref{orthogonal.lem.pf.eq7} proves that the constructed
%$u, v$ satisfy \eqref{hilbertpr.thm.eq5}.
This contradicts to the hypothesis in the sufficiency.
 %assumption and hence it completes the proof of the sufficiency in Theorem \ref{hilbertpr.thm}.

\subsection{Proof of Theorem \ref{subspacePR.thm}}\label{subspacePR.thm.pfsection}
 By $(I-P_{\tilde {\mathcal H}})\pm P_{\tilde {\mathcal H}}\in {\mathcal U}({\mathcal{H}})$,
the unitary invariance of the linear space ${\mathcal S}$  and  the set $\Phi$ of linear measurements, we have
  %$V$ is a linear space and being
\begin{equation}\label{subspacePR.cor.pf.eq1}
P_{\tilde {\mathcal H}}f\in {\mathcal S}
\end{equation}
and
\begin{equation} \label{projective.eq14}
\phi\big((I-2P_{\tilde{\mathcal H}}) f\big)= (I-2P_{\tilde{\mathcal H}}) \phi(f)
\end{equation}
for all $f\in {\mathcal S}$ and $\phi\in \Phi$.
Therefore
\begin{equation}
\label{projective.eq10}P_{\tilde {\mathcal H}} {\mathcal S}=\{h\in {\mathcal S},   \   h(x)\in {\tilde {\mathcal H}} \ {\rm for \ all} \ x\in D\}\subset {\mathcal S}
\end{equation}
and
 \begin{equation}\label{orthognal.projective.eq1}
 \phi(P_{\tilde {\mathcal H}}f)=P_{\tilde {\mathcal H}}\phi(f)\in \tilde{\mathcal H}\ \ {\rm for\ all}\ f\in \mathcal S.
 \end{equation}
%and hence $\phi(P_{\tilde {\mathcal H} }\mathcal S)\subset \tilde {\mathcal H}$.

For any real unitary operator $U \in {\mathcal U}(\tilde{\mathcal H})$, $\phi\in \Phi$  and $g\in P_{\tilde {\mathcal H}} {\mathcal S}$, we have
\begin{equation}\label{projective.eq11}
 Ug(x)\in \tilde {\mathcal H}\ \ {\rm for \ all}\ x\in D,\end{equation}
\begin{equation} \label{projective.eq12}  Ug= \big(U+(I-P_{\tilde{\mathcal H}})\big) g\in  {\mathcal S},
\end{equation}
and
\begin{equation}  \label{projective.eq12+}
\phi(Ug)= \phi\big((UP_{\tilde{\mathcal H}}+I-P_{\tilde{\mathcal H}}) g\big)=
(UP_{\tilde{\mathcal H}}+I-P_{\tilde{\mathcal H}}) \phi(g)= U \phi(g)
%\phi\big((TP_{\tilde{\mathcal H}}+I-P_{\tilde{\mathcal H}}) f\big)= (TP_{\tilde{\mathcal H}}+I-P_{\tilde{\mathcal H}})  \phi(f)
\end{equation}
by \eqref{projective.eq14}, \eqref{orthognal.projective.eq1} and unitary invariance of the linear space  ${\mathcal S}$ and   the set $\Phi$ of linear measurements.
Therefore  $P_{\tilde {\mathcal H}} {\mathcal S}$ is unitary invariant,
\begin{equation}  \label{projective.eq13}
U P_{\tilde {\mathcal H}} {\mathcal S}\subset P_{\tilde {\mathcal H}} {\mathcal S}, \ \  U\in {\mathcal U}(\tilde{\mathcal H}),
\end{equation}
and
the set $\Phi$ of linear measurements is also unitary invariant for the linear space $P_{\tilde {\mathcal H}} {\mathcal S}$
by \eqref{projective.eq10}, \eqref{projective.eq11},  \eqref{projective.eq12} and \eqref{projective.eq12+}.

Take an $f\in {\mathcal S}$. Suppose, on the contrary, that $P_{\tilde {\mathcal H}}f$ is not  phase retrieval. Applying
 Theorem \ref{hilbertpr.thm} to the linear space $P_{\tilde {\mathcal H}} {\mathcal S}$, % \st{of functions on the domain $D$},
 we know  that there exist $u, v\in P_{\tilde {\mathcal H}} {\mathcal S}$ satisfying
 \eqref{hilbertpr.thm.eq3}, \eqref{hilbertpr.thm.eq4} and \eqref{hilbertpr.thm.eq5}. % so that $P_{\mathcal G}f=u+v$.  %Therefore
 Define
 \begin{equation} \label{projective.eq15}
 \tilde u=(I- P_{\tilde {\mathcal H}})f+u\ \ {\rm  and} \ \  \tilde v=v\in {\mathcal S}.\end{equation}
   From    \eqref{hilbertpr.thm.eq3}, \eqref{hilbertpr.thm.eq4}, \eqref{hilbertpr.thm.eq5},
  \eqref{projective.eq14}, \eqref{orthognal.projective.eq1},
  \eqref{projective.eq15} and the projection property of $P_{\tilde {\mathcal H}}$, we have
$ f=\tilde{u}+\tilde{v}$,
 \begin{equation*}\label{invariant.pr.eq2}\langle \phi(\tilde u), \phi(\tilde v)\rangle=
 \langle (I-P_{\tilde {\mathcal H}}) \phi(f), \phi( v)\rangle+\langle \phi(u), \phi( v)\rangle  %=\langle \phi( u), \phi( v)\rangle
 =0\end{equation*}
for all $\phi\in \Phi$, and
% $$\langle \tilde u(x), \tilde v(y)\rangle= \langle  u(x),  v(y)\rangle\ {\rm for \ all} \ x, y\in D.$$
 %{\color{red}According to the definition 3.1 , we have u+(I-P_{\mathcal G})f)
 \begin{eqnarray*}\label{invariant.pr.eq3}
&&\langle \phi_0(\tilde u), \phi_1(\tilde v)\rangle+\langle {\phi}_1(\tilde u), \phi_0(\tilde v)\rangle\nonumber\\
&\hskip-0.08in =& \hskip-0.08in \langle  (I-P_{\tilde {\mathcal H}}) \phi_0(f)+\phi_0(u),\phi_1(v)\rangle
+\langle (I-P_{\tilde {\mathcal H}})\phi_1(f)+\phi_1(u), \phi_0(v)\rangle\\
%&=&\big\langle \phi_0\big((I-P_{\tilde {\mathcal H}})(f)+u\big),\phi_1(v)\big\rangle+\big\langle \phi_1\big((I-P_{\tilde {\mathcal H}})(f)+u\big), \phi_0(v)\big\rangle\nonumber\\
 &\hskip-0.08in =& \hskip-0.08in\langle \phi_0(u), \phi_1(v)\rangle+\langle {\phi}_1(u),\phi_0(v)\rangle\ne 0  %+\langle (I-P_G)f,v\rangle =\langle (I-P_G)f,\phi^\ast\phi(v)\rangle
 \end{eqnarray*}
 for the linear measurements $\phi_0, \phi_1$ in \eqref{hilbertpr.thm.eq4}.
 Therefore $f$ is not  phase retrieval in ${\mathcal S}$ by Theorem \ref{hilbertpr.thm}.
 This contradicts to our hypothesis on the function $f$ and hence completes the proof.
%\end{proof}

\subsection{Proof of Theorem \ref{vectorfield.thm}}\label{vectorfield.thm.section}

To prove Theorem  \ref{vectorfield.thm}, we need a technical lemma  about isomorphism of two $d$-simplices.

\begin{lem}\label{simplex.lem}
Let ${\bf x}_i\in \Rd, 0\le i\le d$, be affinely independent.
If ${\bf y}_i, 0\le i\le d$, satisfy
\begin{equation}\label{simplex.lem.eq1}
\|{\bf y}_i\|=\|{\bf x}_i\|, \ 0\le i\le d
\end{equation}
and
\begin{equation}\label{simplex.lem.eq2}
\|{\bf y}_i-{\bf y}_j\|=\|{\bf x}_i-{\bf x}_j\|, \ 0\le i, j\le d,
\end{equation}
then there exists an orthogonal matrix $ U\in {\mathcal U}(\Rd)$ such that
\begin{equation} \label{simplex.lem.eq3}
{\bf y}_i= U {\bf x}_i, \ 0\le i\le d.
\end{equation}
\end{lem}

\begin{proof}  For the completeness of this paper, we include a proof.  By \eqref{simplex.lem.eq1} and \eqref{simplex.lem.eq2}, we have
\begin{equation} \label{simplex.lem.pfeq1}
\langle {\bf y}_{i}, {\bf y}_{j}\rangle= \langle {\bf x}_{i}, {\bf x}_{j}\rangle,\  0\le i, j\le d.
\end{equation}
Let $ U$ be the linear transform on $\Rd$ determined by
\begin{equation} \label{simplex.lem.pfeq2}
 U({\bf x}_{i}-{\bf x}_{0})= {\bf y}_i-{\bf y}_{0}, \ 1\le i\le d,\end{equation}
which is well-defined by the linear independence of ${\bf x}_{i}-{\bf x}_{0}, 1\le i\le d$. For any ${\bf x}=\sum_{i=1}^d c_i ({\bf x}_{i}-{\bf x}_{0})\in \Rd$, we have
\begin{eqnarray} \label{simplex.lem.pfeq2++}  \|U{\bf x}\|^2  & \hskip-0.08in = & \hskip-0.08in
 \Big\|\sum_{i=1}^d c_i ({\bf y}_{i}-{\bf y}_{0})\Big\|^2=
\sum_{i, j=1}^d c_i c_j \langle  {\bf y}_{i}-{\bf y}_{0}, {\bf y}_{j}-{\bf y}_{0}\rangle\nonumber\\
& \hskip-0.08in = & \hskip-0.08in
\sum_{i, j=1}^d c_i c_j \langle  {\bf x}_{i}-{\bf x}_{0}, {\bf x}_{j}-{\bf x}_{0}\rangle
=\|{\bf x}\|^2
\end{eqnarray}
by \eqref{simplex.lem.pfeq1} and \eqref{simplex.lem.pfeq2}.
This proves that $U$ is  an orthogonal matrix in ${\mathcal U}(\Rd)$.
Therefore by \eqref{simplex.lem.pfeq2}  and \eqref{simplex.lem.pfeq2++} it suffices to prove that
${\bf y}_0= U {\bf x}_0$, or equivalently
\begin{equation} \label{simplex.lem.pfeq3}
{\bf x}_0=  U^T {\bf y}_0.
\end{equation}
%By the linear independence of ${\bf x}_{i}-{\bf x}_{0}, 1\le i\le d$, there exist a unique solution  $(a_{j})_{1\le j\le d}$ to the linear system
%$$\sum_{j=1}^d  \langle {\bf x}_i-{\bf x}_0, {\bf x}_j-{\bf x}_0 \rangle a_j=\langle {\bf x}_i-{\bf x}_0,  {\bf x}_0\rangle,  1\le i\le d.$$
%For the above solution,  one may verify that
%\begin{equation}
%{\bf x}_0=\sum_{j=1}^d a_{j} ({\bf x}_i-{\bf x}_0)
%\ {\rm and } \ {\bf y}_0=\sum_{j=1}^d a_{j} ({\bf y}_j-{\bf y}_0),
%\end{equation}
%which together with \eqref{simplex.lem.pfeq2} proves \eqref{simplex.lem.pfeq3} and completes the proof.
 By \eqref{simplex.lem.pfeq1}, \eqref{simplex.lem.pfeq2} and \eqref{simplex.lem.pfeq2++}, we have
$$ %\langle  U({\bf x}_{i}-{\bf x}_{0}), U {\bf x}_{0}\rangle
\langle{\bf x}_{i}-{\bf x}_{0}, {\bf x}_{0}\rangle=\langle{\bf y}_{i}-{\bf y}_{0}, {\bf y}_{0}\rangle=\langle U({\bf x}_{i}-{\bf x}_{0}), {\bf y}_{0}\rangle=
\langle {\bf x}_{i}-{\bf x}_{0},  U^T {\bf y}_{0}\rangle
$$
for all $1\le i\le d$.
This together with   the linear independence of ${\bf x}_{i}-{\bf x}_{0}, 1\le i\le d$,
%$$\langle U({\bf x}_{i}-{\bf x}_{0}), U{\bf x}_{0}-{\bf y}_0\rangle=0, \ 1\le i\le d. $$
%  one may verify that $U({\bf x}_{i}-{\bf x}_{0}), 1\le i\le d$ are linearly independent.
%
%  which implies
%
%This together with  $U({\bf x}_{i}-{\bf x}_{0}), 1\le i\le d$ are linearly independent
proves \eqref{simplex.lem.pfeq3}
  and completes the proof.
\end{proof}

\begin{proof}[Proof of Theorem \ref{vectorfield.thm}]
Let ${\mathcal G}_1=(V_1, E_1)$ be a complete subgraph  of order $d+1$ such that
$\triangle({\bf f}, {\mathcal G}_1)$ is a $d$-simplex. Let ${\bf g}=({\bf g}_i)_{i\in V}$ be a vector field satisfying \eqref{vectorfield.thm.pfeq1}, \eqref{vectorfield.thm.pfeq2}. By
%By  \eqref{vectorfield.thm.pfeq1}, \eqref{vectorfield.thm.pfeq2} and
Lemma \ref{simplex.lem},  there exists
 an orthogonal matrix $U$ such that
\begin{equation}  \label{vectorfield.thm.pfeq3}
U{\bf f}_i={\bf g}_i, i\in V_1.
\end{equation}
Define $\tilde  {\bf  g}= U^{T} {\bf g}$. One may verify that the new vector field $\tilde {\bf g}=(\tilde {\bf g}_i)_{i\in V}$ still satisfies
\eqref{vectorfield.thm.pfeq1} and \eqref{vectorfield.thm.pfeq2} and moreover, it coincides with the original vector field on
vertices in the subgraph ${\mathcal G}_1$,
\begin{equation}  \label{vectorfield.thm.pfeq4}
\tilde {\bf g}_i= {\bf f}_i, i\in V_1.
\end{equation}
Now we prove that the vector fields $\tilde {\bf g}$ and ${\bf f}$ are the same, i.e.,
\begin{equation}
\tilde {\bf g}_i= {\bf f}_i, i\in V.
\end{equation}
By the  covering property for  $d$-simplices, and the connectivity of   $d$-simplex graph ${\mathcal G}_{\bf f}$,
it suffices to prove that
\begin{equation}\label{equationforV1V2}
\tilde {\bf g}_i={\bf f}_i, i\in  V_1\cup V_2,
\end{equation}
where
 ${\mathcal G}_2=(V_2, E_2)$ be a complete subgraph of order $d+1$ such that
 $\triangle({\bf f}, {\mathcal G}_2)$ is a neighboring $d$-simplex of $\triangle({\bf f}, {\mathcal G}_1)$.
By the assumption on the $d$-simplex  $\triangle({\bf f}, {\mathcal G}_2)$,
 we have that $V_1\cap V_2$ has $d$ vertices and ${\bf f}_i, i\in V_1\cap V_2$,  form a basis for $\Rd$, see Remark \ref{simplegraph.rem}.
For all $i\in V_1\cap V_2$ and $k\in V_2\backslash V_1$, we obtain from  \eqref{vectorfield.thm.pfeq1}, \eqref{vectorfield.thm.pfeq2} and
\eqref{vectorfield.thm.pfeq4} that
\begin{eqnarray}\label{vector.temp0}
\langle \tilde {\bf g}_k,  {\bf f}_i\rangle &\hskip-0.08in  = & \hskip-0.08in  \langle \tilde {\bf g}_k, \tilde {\bf g}_i\rangle = \frac{1}{2} (\|\tilde {\bf g}_k\|^2+ \|\tilde {\bf g}_i\|^2-
\|\tilde {\bf g}_k-\tilde {\bf g}_i\|^2)\nonumber\\
& \hskip-0.08in = & \hskip-0.08in \frac{1}{2} (\|{\bf f}_k\|^2+ \|\tilde {\bf f}_i\|^2-
\|\tilde {\bf f}_k-\tilde {\bf f}_i\|^2)=
\langle {\bf f}_k, {\bf f}_i\rangle.
\end{eqnarray}
By \eqref{vector.temp0} and the basis property for ${\bf f}_i, i\in V_1\cap V_2$, we obtain $\tilde {\bf g}_k= {\bf f}_k$.
This proves \eqref{equationforV1V2} and completes the proof.
\end{proof}

\subsection{Proof of Theorem \ref{complexpr.thm}} \label{complexpr.thm.section}
Let ${\mathcal C}_{\R^2}$ be as in \eqref{cr2.def}, and
define  the linear map $A:{\mathcal C} \longmapsto {\mathcal C}_{\R^2}$ by
$$Af=\begin{pmatrix} \Re f \\ \Im f
 \end{pmatrix}, \ f\in {\mathcal C}.$$  Then  $A$ is one-to-one and onto,  and
 \begin{equation}\label{afg.def0}
 \langle Af(x), Ag(y)\rangle= \Re(f(x) \bar g(y))
 \end{equation}
 for all $f, g\in {\mathcal C}$ and  $x, y\in D$. Also one may verify that
  $$Ag=U (Af)$$ for some orthogonal matrix $U$ of size $2\times 2$ if and only if
  either $ g=zf$ or  $g=z\bar f$ for some $z\in \T$.
This together with the complex conjugate invariance
of the complex linear space ${\mathcal C}$ implies that the real linear space  ${\mathcal C}_{\R^2}$ is
invariant under all real  orthogonal transformations on $\R^2$.
Therefore by \eqref{afg.def0} and Theorem \ref{hilbertpr.thm}, the proof  of Theorem \ref{complexpr.thm} reduces to establishing
Assumption \ref{unitary.assumption} for the family  $\Phi=\{\delta_x, x\in D\}$ of
the point-evaluation measurements on ${\mathcal C}_{\R^2}$, which can be reformulated as the equivalence
between the following two statements for $f, g\in {\mathcal C}$:
\begin{itemize}
\item[{(i)}] $g=zf$ or  $g=z\bar f$ for  some $z\in \T$. %.

\item[{(ii)}]
  $\Re (f(x) \bar f(y)-g(x)\bar g(y))=0$ for all $x, y\in D$.
\end{itemize}
The implication  (i)$\Longrightarrow$(ii) is obvious. Now we prove that (ii)$\Longrightarrow$(i).
Applying (ii) with $y$ replaced by $x\in D$, we have
\begin{equation}\label{new5.4}
|g(x)|^2=|f(x)|^2, x\in D.
\end{equation}
Therefore the conclusion (i) follows from \eqref{new5.4} if $f$ is a zero function. For a nonzero function $f$, without loss of generality, there exists $x_0\in D$ such that
%$f(x_0)\ne 0$.  Without loss of generality, we assume that
\begin{equation} \label{new5.5} 0\ne f(x_0)=g(x_0)\in \R
\end{equation}
by \eqref{new5.4}, otherwise replacing $f$ by $z_1 f$ and $g$ by  $z_2 g$ respectively, where $z_1, z_2\in \T$.
Applying  (ii) with $y$ replaced by $x_0$, we obtain that
\begin{equation}   \label{new5.6}
\Re f(x)= \Re g(x), x\in D.
\end{equation}
 Set \begin{equation}\label{D1D2.def0}
 D_1=\{x\in D, \   g(x)= f(x)\}\ {\rm  and} \ D_2= \{x\in D, \   g(x)= \bar f(x)\}.\end{equation}
 Then
  \begin{equation}\label{D1D2D.new} D_1\cup D_2= D.
 \end{equation}
 by \eqref{new5.4}  and \eqref{new5.6}.
 The statement (i) is proved if either $D_1=D$ or $D_2=D$.  Now we consider the case that
 \begin{equation} \label{D1D2.new} D_1\ne D\  \ {\rm and} \ \  D_2\ne D.
 \end{equation}
 By \eqref{D1D2D.new} and \eqref{D1D2.new}, $D_1\backslash  D_2\ne \emptyset$ and $D_2\backslash D_1\ne \emptyset$.
  %$ x_1\in D_1\setminus D_2\subset D$ and $y_1\in D_2\setminus D_1\subset D$, we have
 Taking $ x_1\in D_1\backslash  D_2$ and $y_1\in D_2\backslash D_1$, we obtain from \eqref{D1D2.def0} that
 \begin{eqnarray*}\Re (f(x_1) \bar f(y_1)-g(x_1)\bar g(y_1)) & \hskip-0.08in = & \hskip-0.08in  \Re \big(f(x_1) (\bar  f(y_1)- f(y_1))\big)\nonumber\\
& \hskip-0.08in = & \hskip-0.08in  2\Im f(x_1) \Im f(y_1)\ne 0,
 \end{eqnarray*}
 which contradicts to the statement (ii).

\subsection{Proof of Theorem \ref{necess.charac}} \label{necess.charac.section}
 Take   a nonzero function $g$ in the real linear space spanned by the real and imaginary parts of $f$.   Then
there exists a  nonzero complex number $z$ such that
$g=zf+\bar z \bar f$.
Without loss of generality, we assume that
$z=1$,
as
$ z f$
is conjugate phase retrieval in ${\mathcal C}$. %  as $\bar z\bar f$ is invariant under complex conjugation in $\mathcal C$.
Suppose, on the contrary,  that  $g=f+\bar f$ is not   phase retrieval in $\Re({\mathcal C})$.
Then  there exist nonzero functions $f_1, f_2\in \Re({\mathcal C})$  by \eqref{realnonseparable} such that
 \begin{equation}  \label{necess.charac.pf.eq1}
g= f_1+f_2\ \  {\rm and} \ \  f_1f_2=0.
 \end{equation}
 Write
 \begin{equation} \label {necess.charac.pf.eq2+}f=u+v,
 \end{equation}
 where
 \begin{equation} \label{necess.charac.pf.eq2}
 u= \frac{f-\bar f+ f_1}{2} \ {\rm and} \  v=\frac{f_2}{2}.
 \end{equation}
 By  \eqref{necess.charac.pf.eq1} and \eqref {necess.charac.pf.eq2}, we have
 \begin{equation} \label{necess.charac.pf.eq3}
 \Re (u(x) \bar v(x))= \frac{1}{4} \Re \big((f(x)-\bar f(x)+f_1(x)) f_2(x)\big)= \frac{1}{4} f_1(x)f_2(x)=0
 \end{equation}
 for all $x\in D$,
 where the  second equality holds  as  $ i(f(x)-\bar f(x)), f_1(x), f_2(x)\in \R$ for all $x\in D$.
 Take $x_0, y_0\in D$ such that
 \begin{equation} \label{necess.charac.pf.eq4}
 f_1(x_0)\ne 0\ {\rm and} \ f_2(y_0)\ne 0.\end{equation}
 Then
 \begin{eqnarray} \label{necess.charac.pf.eq5}
\Re(u(x_0) \bar v(y_0)+u(y_0) \bar v(x_0)) & \hskip-0.08in = & \hskip-0.08in \frac{ f_1(x_0) f_2 (y_0)+f_1(y_0)  f_2(x_0)}{4}\nonumber\\
& \hskip-0.08in = & \hskip-0.08in
\frac{f_1(x_0) f_2 (y_0)}{4}\ne 0
\end{eqnarray}
by \eqref{necess.charac.pf.eq1}, \eqref {necess.charac.pf.eq2} and \eqref{necess.charac.pf.eq4}.
By \eqref{necess.charac.pf.eq2+}, \eqref{necess.charac.pf.eq3}, \eqref{necess.charac.pf.eq5} and Theorem \ref{complexpr.thm},
we obtain that $f$ is not  complex conjugate  phase retrieval, which is a contradiction.
%\end{proof}

\subsection{Proof of Proposition \ref{mtwopr.prop}}\label{mtwopr.prop.section}
 The sufficiency follows from  Corollary \ref{matrixprojection.cor}.
Now we prove the necessity.   Write ${\bf A}^T=({\bf a}_1, \ldots, {\bf a}_m)$.
By the assumption, the matrix ${\bf A}$ has the complement property \cite{BCE06}.  %{\color{red} reference or detailed information}
Hence the matrix  ${\bf A}$ has rank $2$ and
 the linear space of all real symmetric matrices is spanned by
${\bf a}_i {\bf a}_j^T+ {\bf a}_j {\bf a}_i^T, 1\le i, j\le m$.
By Corollary \ref{corollary2.4}, it suffices to prove that
${\bf a}_i {\bf a}_j^T+ {\bf a}_j {\bf a}_i^T, 1\le i, j\le m$, are
linear combinations of
${\bf a}_k {\bf a}_k^T, 1\le k\le m$.
By the complement property, without loss of generality, we assume that
$\span \{{\bf a}_1, {\bf a}_2\}=\R^2$
and $ {\bf a}_3=\lambda {\bf a}_1+\mu {\bf a}_2$ for some nonzero real numbers $\lambda$ and $\mu$.
Then
\begin{equation}\label{a3.eq}
{\bf a}_1 {\bf a}_2^T+  {\bf a}_2{\bf a}_1^T=(\lambda\mu)^{-1} {\bf a}_3 {\bf a}_3^T- { \lambda\mu^{-1}}
{\bf a}_1 {\bf a}_1^T- {\lambda^{-1}\mu} {\bf a}_2 {\bf a}_2^T.
\end{equation}
By the assumption on ${\bf a}_1$ and ${\bf a}_2$, there exist $\lambda_i, \mu_i, 1\le i\le m$ such that
${\bf a}_i=\lambda_i {\bf a}_1+\mu_i {\bf a}_2$, which together with \eqref{a3.eq} implies that
\begin{eqnarray*}
 {\bf a}_i {\bf a}_j^T+ {\bf a}_j {\bf a}_i^T  & \hskip-0.08in =  & \hskip-0.08in 2 \lambda_i \lambda_j {\bf a}_1 {\bf a}_1^T
+ (\lambda_i \mu_j+\lambda_j \mu_i) ({\bf a}_1 {\bf a}_2^T+  {\bf a}_2{\bf a}_1^T)+
2 \mu_i \mu_j {\bf a}_1 {\bf a}_1^T\nonumber\\
& \hskip-0.08in \in  & \hskip-0.12in   \span \{ {\bf a}_l{\bf a}_l^T, l=1, 2,3\},\ \ 1\le i, j\le m.
\end{eqnarray*}
This completes the proof.
%for all $1\le i, j\le m$.

\subsection{Proof of Proposition \ref{complex.sis.pr}}\label{complex.sis.pr.pfsection}
To prove Proposition \ref{complex.sis.pr}, we recall a characterization for a  phase retrieval  function in $S(h)$.
  %\st{for a function in $V_\R(h)$ to be real phase retrieval.}

\begin{lem}\label{complex.sis.lem}{\rm \cite[Theorem 3.2]{YC16}}
Let $g(t)=\sum_{k\in \Z} d(k) h(t-k)\in S(h),  d(k) \in \R$. Then $g$ is  phase retrieval in $S(h)$ if and only if
\begin{equation}  \label{complex.sis.lem.eq1}
d(k)\ne 0  \ {\rm for \ all} \ K_--1< k<K_++1\end{equation}
where $K_-=\inf\{k, \ d(k)\neq 0\}$ and $K_+=\sup\{k, \  d(k)\neq 0\}$.
\end{lem}

\begin{proof}[Proof of Proposition \ref{complex.sis.pr}]  $\Longrightarrow$: \quad  Suppose, on the contrary, that \eqref{complex.sis.pr.eq2} does not hold.
Then  there exist $k_0, k_1, k_2\in (K_--1, K_++1)$ such that
\begin{equation}\label{complex.sis.pr.pfeq1}
k_0<k_1<k_2, \  c(k_0)c(k_2)\ne 0 \ \ {\rm  and} \ \ c(k_1)=0.
\end{equation}
Take $z_1\in \T$ with
\begin{equation}\label{complex.sis.pr.pfeq2} \Re(z_1 c(k_0))\ne 0 \ {\rm and} \  \Re(z_1 c(k_2))\ne 0.
\end{equation}
Then
$$\Re(z_1f(t))=\sum_{k<k_1} \Re(z_1 c(k)) h(t-k)+ \sum_{k>k_1} \Re(z_1 c(k)) h(t-k)$$
is not  phase retrieval in $S(h)$ by Lemma \ref{complex.sis.lem}, which is a contradiction by Theorem \ref{necess.charac}.
This completes the proof of \eqref{complex.sis.pr.eq2}.

%Set $u= \frac{1-z_0}{2} \sum_{k>k_1} c(k) h(t-k)$ and $v= f-u$, where
%$z_0\in \C$ is so chosen that
%\begin{equation}\label{z_0.assump} |z_0|=1\ \  {\rm and} \  \   \Re( (1-z_0) \bar c(k_0)c(k_2))\ne 0.\end{equation}
%Then  $f=u+v$ and
%\begin{equation} \Re (u(k_2) \bar v(k_0)+ u(k_0)\bar v(k_2))= \Re (u(k_2) \bar v(k_0))=
%\Re\Big(\frac{1-z_0}{2} c(k_2) \bar c(k_0)\Big)\ne 0
%\end{equation}
%by \eqref{z_0.assump} and the interpolating property of the hat function $h$ on $\Z$. Also, we have
%\begin{eqnarray}
%\Re(u(t)\bar v(t)) & \hskip-0.08in = & \hskip-0.08in \Re \Big(\frac{1-z_0}{2}  \sum_{k>k_1} c(k) h(t-k) \times  \sum_{k<k_1} c(k) h(t-k)\Big)\nonumber\\
%& \hskip-0.08in & \hskip-0.08in +
%\Re \Big(\frac{1-z_0}{2}  \sum_{k>k_1} c(k) h(t-k) \times \frac{1+\bar z_0}{2} \sum_{k>k_1} \bar c(k) h(t-k)\Big)\nonumber\\
%& \hskip-0.08in = & \hskip-0.08in \Re \Big(\frac{(1-z_0)(1+\bar z_0)}{4} \Big) \Big|\sum_{k>k_1}  c(k) h(t-k)\Big|^2=0, \ t\in \R,
%\end{eqnarray}
%where the second equality from disjoint supporting property for  the functions $\sum_{k<k_1} c(k) h(t-k)$ and  $\sum_{k>k_1} \bar c(k) h(t-k)$,
%and the third equality holds by \eqref{z_0.assump}.
% This leads to a contradiction by  Theorem \ref{complexpr.thm} and hence
% \eqref{complex.sis.pr.eq2} is proved.

For any $l_0\in (K_--1, K_++1)$, it follows from Theorem \ref{necess.charac}
that
$$\Re\big(i \bar c(l_0) f(t)\big)=\Big(\sum_{k<l_0}+\sum_{k>l_0}\Big) \Re\big(i \bar c(l_0) c(k)\big) h(t-k)$$
is  phase retrieval in $S(h)$. Hence by Lemma \ref{complex.sis.lem}, either
$\Re\big(i \bar c(l_0) c(k)\big)=0$ for all $k<l_0$ or
$\Re\big(i \bar c(l_0) c(k)\big)=0$ for all $k>l_0$.  This proves \eqref{complex.sis.pr.eq3}.

$\Longleftarrow$: \quad  We prove the sufficiency by two cases.
\smallskip

{\bf Case 1}: \  $\Im \big(c(k)\overline{c(k+1)}\big)=0$ for all $k\in (K_--1, K_+)$.

Without loss of generality, we assume that $c(k)\in \R$ for all $k\in \Z$, otherwise replacing $f$ by $zf$ for some  $z\in \T$. %{\color{red}
%Then $f=\sum_{k\in \Z}c(k)h(\cdot-k)$ is phase retrieval follows immediately by Theorem 3.2 in \cite{YC16}.
%Let me know if you want to use this statement or not, and I can make suitable adjustment for the proof of Case 2. }
Let  $g\in S_{\C}(h)$ such that
\begin{equation}\label{gf.equation0}
|g(t)|= |f(t)|, \  t\in \R.
\end{equation}
Take $k_0\in \Z$ with  $f(k_0)\ne 0$. Without loss of generality, we assume that
\begin{equation}\label{gko.eq} g(k_0)=f(k_0),\end{equation} otherwise replacing  $g$ by $\tilde g=\frac{f(k_0)}{g(k_0)} g$.
By \eqref{gf.equation0} with $t$ replaced by $k_0+u\in k_0+[0, 1]$,
\begin{equation*}
|f(k_0) (1-u)+ f(k_0+1) u|^2= |g(k_0) (1-u) + g(k_0+1) u|^2, u\in [0, 1].
\end{equation*}
This together with $0\ne g(k_0)=f(k_0)\in \R$ and $f(k_0+1)\in \R$ implies that
$ g(k_0+1)=f(k_0+1)$.
Applying the above argument inductively, we have
 \begin{equation}\label{gko.eq1}
 g(k)=f(k) \ {\rm for \ all} \  k\in (k_0, K_++1). \end{equation}

Applying \eqref{gf.equation0} with $t$ replaced by $k_0-u \in k_0+[-1, 0]$ and
using $0\ne g(k_0)=f(k_0)\in \R$ and $f(k_0-1)\in \R$, we can show that
%we can prove that
%\begin{equation*}
%|f(k_0-1) u+ f(k_0) (1-u)|^2= |g(k_0) u + g(k_0+1) (1-u)|^2, u\in [0, 1].
%\end{equation*}
%This together with $0\ne g(k_0)=f(k_0)\in \R$ and $f(k_0-1)\in \R$ implies that
$ g(k_0-1)=f(k_0-1)$.
Thus we have the following conclusion by  induction,
 \begin{equation}\label{gko.eq2}
  g(k)=f(k), k\in (K_--1, k_0).\end{equation}
 By \eqref{gf.equation0},  we have
  \begin{equation}\label{gko.eq3}
  g(k)=f(k)=0, k\not\in (K_--1, K_++1). \end{equation}
 Combining \eqref{gko.eq}--\eqref{gko.eq3} proves $g=f$, which completes the proof.

  \smallskip
  {\bf Case 2}: \  There exists $k_0\in (K_--1, K_+)$ such that
   $\Im \big(c(k_0)\overline{c(k_0+1)}\big)\ne 0$ and  $\Im \big(c(k)\overline{c(k+1)}\big)=0$ for all $k\in (K_--1, K_+)\setminus \{k_0\}$.

Without loss of generality, we assume that $c(k)\in \R$ for all $k\le k_0$, otherwise replacing $f$ by $zf$ for some $z\in \T$.
Let  $g\in S_{\C}(h)$ such that \eqref {gf.equation0} holds.
Without loss of generality, we assume that $0\neq g(k_0)=f(k_0)\in \R$, otherwise replacing  $g$ by $\tilde g=\frac{f(k_0)}{g(k_0)} g$.
Following the argument used in Case 1, we can show that
\begin{equation}\label{gko.eq4}
g(k)=f(k)\in \R,  \ k\le k_0.
\end{equation}
% By  with
%\begin{equation*}
%|f(k_0)(1- u)+ f(k_0+1) u|^2= |g(k_0)(1- u )+ g(k_0+1) u|^2, u\in [0, 1],
%\end{equation*}
%with together with  }
Similarly,  with $t$ replaced by $k_0+u\in k_0+[0, 1]$ in \eqref{gf.equation0}  and $0\ne g(k_0)=f(k_0)\in \R$, we obtain
\begin{equation*}
\Re g(k_0+1)= \Re f(k_0+1) \ {\rm and} \ |g(k_0+1)|=|f(k_0+1)|.
\end{equation*}
Hence either
\begin{equation}\label {gko.eq5}
g(k_0+1)= f(k_0+1)
\end{equation}
or
\begin{equation} \label {gko.eq6}
g(k_0+1)= \bar f(k_0+1).
\end{equation}
For the case that \eqref{gko.eq5} holds, we can
follow the argument used in Case 1 to prove
\begin{equation}\label {gko.eq7}
g(k)= f(k), k\ge k_0+1.
\end{equation}
Similarly for the case that \eqref{gko.eq6} holds, we have
\begin{equation}\label {gko.eq8}
g(k)=\bar  f(k), k\ge k_0+1.
\end{equation}
Combining \eqref{gko.eq4}, \eqref{gko.eq7} and \eqref{gko.eq8}, we conclude that
either
$g(k)=f(k)$ for all $k\in \Z$ or $g(k)=\bar f(k)$ for all $k\in \Z$.
Therefore either $g=f$ or $g=\bar f$. This completes the proof of conjugate phase retrieval for  Case 2 and  completes the proof of sufficiency.
\end{proof}

\subsection{Proof of Theorem \ref{quaternionpr.thm}} \label{quaternionpf.thm.pfsection}
%\begin{proof} %{\color{red} Chen: now I believe that you can have a complete proof now}

%For a signal ${\bf f}=f_1+f_2{\bf i}+f_3{\bf j}+f_4{\bf k}\in\mathcal W$, denote  by ${\bf v}_f=(f_1, f_2, f_3, f_4)'\in \R^4$ and
% Due to the isomporphism between $Q_8$ and $\R^4$,
%$$z=a+bi+cj+d k\longmapsto (a, b, c, d)^T\in \R^4$$
%By Theorem ???, it suffices to verify that

The Conclusion (ii) follows from the first conclusion. Then it suffices to prove the Conclusion (i).
 Let the linear subspace  $\mathcal W_\R \subset {\mathcal W}$ contain all real functions in $\mathcal W$, and ${\mathcal W}_\R^4$ be the product space of $\mathcal W_\R$. Obviously, ${\mathcal W}_\R^4$ is  unitary invariant.
By  \eqref{f1234.def}  and the isomporphism
$$z=a+b{\bf i}+c{\bf j}+d {\bf k}\longmapsto (a, b, c, d)^T\in \R^4$$
between $Q_8$ and $\R^4$,
 the map   $T_{\mathcal W}$ defined by
$$  T_{\mathcal W}: {\mathcal W}_\R^4\ni (f_1, f_2, f_3, f_4)^T\longmapsto f_1+f_2{\bf i}+f_3 {\bf j}+f_4{\bf k}\in {\mathcal W}$$
is an isomorphism between ${\mathcal W}_\R^4$ and ${\mathcal W}$,
 and
for any ${\bf f}, {\bf g}\in {\mathcal W}_\R^4$,
$$ \langle {\bf f}(x), {\bf g}(y)\rangle= \Re \big(T_{\mathcal W}({\bf f})(x) T_{\mathcal W}({\bf g})^\ast(y)\big),  x, y\in D.
$$
Then it suffices to  prove that
 ${\bf f}=(f_1, f_2, f_3, f_4)^T\in {\mathcal W}_\R^4$ is
phase retrieval in ${\mathcal W}_\R^4$ if and only if $T_{\mathcal W}(\bf f)$ is quaternion conjugate phase retrieval in ${\mathcal W}$, which in turn reduces to prove
\begin{equation}\label{abcd0.eq}
\{ T_{\mathcal W}(U{\bf f}), U\in {\mathcal U}(\R^4)\}= \Big\{ \sum_{i=1}^4 q_i f_i\in  {\mathcal W},
q_i\in {\T}_8, 1\le i\le 4, \ {\rm satisfies\ \eqref{qiqj.def}}\Big\}.
\end{equation}
Write $U=(u_{ij})_{1\le i, j\le 4}$, then
\begin{eqnarray*} % \label {abcd0.eq}
 T_{\mathcal W}(U{\bf f})\hskip-0.05in & = \hskip-0.05in&\hskip-0.05in  \sum_{j=1}^4 u_{1j}f_j+ \sum_{j=1}^4 u_{2j}f_j {\bf i}+ \sum_{j=1}^4 u_{3j}f_j {\bf j}+\sum_{j=1}^4 u_{4j}f_j {\bf k}
\nonumber\\
\hskip-0.05in& =\hskip-0.05in & \hskip-0.05in \sum_{j=1}^4 (u_{1j}+u_{2j}{\bf i}+ u_{3j} {\bf j}+ u_{4j}{\bf k}) f_j=:\sum_{j=1}^4 p_j f_j.
\end{eqnarray*}
As $U$ is an orthogonal matrix, we have
$$\|p_j\|^2= |u_{1j}|^2+ |u_{2j}|^2+|u_{3j}|^2+|u_{4j}|^2=1$$
for all $1\le j\le 4$,
and
$$p_i p_j^*+ p_j p_i^*= 2 \Re(p_i p_j^*)=2(u_{1i}u_{1j}+u_{2i}u_{2j}+u_{3i}u_{3j}+u_{4i}u_{4j})=0
$$
for $1\le i<j\le 4$.  This proves that
\begin{equation}\label {abcd1.eq}
\{ T_{\mathcal W}(U{\bf f}), U\in {\mathcal U}(\R^4)\}\subset \Big\{ \sum_{i=1}^4 q_i f_i\in  {\mathcal W},
q_i\in {\T}_8, 1\le i\le 4, \ {\rm satisfies\ \eqref{qiqj.def}}\Big\}.
\end{equation}

Let $ q_i\in {\T}_8, 1\le i\le 4$ satisfy \eqref{qiqj.def}.
Write $q_i=u_{1i}+ u_{2i} {\bf i}+ u_{3i} {\bf j}+ u_{4i} {\bf k}, 1\le i\le 4$.
Then it follows from the assumption on $q_i, 1\le i\le 4$ that
\begin{equation*}
u_{1i} u_{1j}+ u_{2i} u_{2j}+u_{3i} u_{3j}+u_{4i} u_{4j}=\left\{\begin{array} {ll} 1 & j=i\\
0 & j\ne i.\end{array}\right.
\end{equation*}
Hence $U=(u_{ij})_{1\le i, j\le 4}$ is  an orthogonal matrix on $\R^4$.
Moreover, we  have that
$$\sum_{i=1}^4 q_i f_i= \Big(\sum_{j=1}^4u_{1j} f_j\Big)+ \Big(\sum_{j=1}^4u_{2j} f_j\Big){\bf i}+
\Big(\sum_{j=1}^4u_{3j} f_j\Big){\bf j}+ \Big(\sum_{j=1}^4u_{4j} f_j\Big){\bf k}= T_{\mathcal W}(U{\bf f}).
$$
This proves that
\begin{equation}\label{abcd2.eq}
\Big\{ \sum_{i=1}^4 q_i f_i\in  {\mathcal W},
q_i\in {\T}_8, 1\le i\le 4, \ {\rm satisfies\ \eqref{qiqj.def}}\Big\}\subset \{ T_{\mathcal W}(U{\bf {\bf f}}), U\in {\mathcal U}(\R^4)\}.
\end{equation}
Combining \eqref{abcd1.eq} and \eqref{abcd2.eq} proves  \eqref{abcd0.eq} and  completes the proof.

\subsection{Proof of Theorem \ref{affinepfmultiple.thm}}\label{affinepfmultiple.thm.pfsection}
For we prove the sufficiency.  Take $f\in {\mathcal S}$ and $g\in  {\mathcal A}_f$.
 Then  for all $1\le i, j\le N$ and $\phi\in \Phi$, we have
%\begin{eqnarray*}
%\langle \phi(g), b_{\phi, i}-b_{\phi,j}\rangle & \hskip-0.08in =
%&  \hskip-0.08in \|\phi(g)+a_{i,\phi}\|^2-\|\phi(g)+b_{\phi, j}\|^2\\
%&  \hskip-0.08in  =&   \hskip-0.08in  \|\phi(f)+a_{i,\phi}\|^2-\|\phi(f)+b_{\phi, j}\|^2
%=\langle \phi(f), b_{\phi, i}-b_{\phi,j}\rangle.
%\end{eqnarray*}
\begin{eqnarray*}
\langle \phi(g), b_{\phi, i}-b_{\phi,j}\rangle & \hskip-0.08in =
&  \hskip-0.08in \frac{\|\phi(g)+b_{\phi,i}\|^2-\|\phi(g)+b_{\phi, j}\|^2-\|b_{\phi,i}\|^2-\|b_{\phi,j}\|^2}{2}\\
&  \hskip-0.08in  =&   \hskip-0.08in  \frac{\|\phi(f)+b_{\phi, i}\|^2-\|\phi(f)+b_{\phi, j}\|^2-\|b_{\phi,i}\|^2-\|b_{\phi,j}\|^2}{2}\\
&  \hskip-0.08in  =&   \hskip-0.08in \langle \phi(f), b_{\phi, i}-b_{\phi,j}\rangle.
\end{eqnarray*}
Therefore by the injectivity of the map $T$ in \eqref{Tijaffine.def}, we obtain that $g=f$, which proves the
affine phase retrieval of the linear space  ${\mathcal S}$.

 Now we prove the necessity. Let $f_0\in {\mathcal S}$ and $1\le i_0\le N$ be as in \eqref{bphiio.def}.
 Suppose, on the contrary, that the linear map $T$ in \eqref{Tijaffine.def} is not injective. Then there exists a nonzero function $h_0\in {\mathcal S}$ such that
 \begin{equation}\label{affinepfmultiple.thm.pf.eq2}
 \langle \phi(h_0), b_{\phi, i}-b_{\phi, j} \rangle=0 \ {\rm for \ all} \ \phi\in \Phi \ {\rm and} \ 1\le i, j\le N.
  \end{equation}
 Set
 $f=-f_0+h_0/2$ and $g=-f_0-h_0/2\in {\mathcal S}$.
 Then $g\ne f$ and $g\in {\mathcal N}_f$, since for all $\phi\in \Phi$ and $1\le i\le N$,
\begin{eqnarray}
& & \|\phi(f)+b_{\phi, i}\|^2- \|\phi(g)+b_{\phi, i}\|^2\nonumber \\
& = &  \|\phi(h_0)/2+ b_{\phi, i}- b_{\phi, i_0}\|^2-
\|-\phi(h_0)/2+ b_{\phi, i}- b_{\phi, i_0}\|^2\nonumber\\
& = & 2 \langle \phi(h_0), b_{\phi, i}- b_{\phi, i_0}\rangle=0
\end{eqnarray}
by \eqref{bphiio.def} and \eqref{affinepfmultiple.thm.pf.eq2}. This contradicts to the affine phase retrieval
of the linear space ${\mathcal S}$.

\end{document}